\documentclass[12pt]{article}
\usepackage{amssymb,amsfonts,amsmath,amsthm}
\usepackage{mathrsfs}
\usepackage{mathtools}
\usepackage{color}
\usepackage{graphicx}
\usepackage{float}
\usepackage{placeins}
\usepackage[ruled,vlined]{algorithm2e}
\usepackage[colorlinks = true,
linkcolor = red,
urlcolor  = red,
citecolor = blue,
anchorcolor = red]{hyperref}
\numberwithin{equation}{section}

\newtheorem{theorem}{Theorem}[section]
\newtheorem{lemma}[theorem]{Lemma}

\usepackage{soul}
\usepackage{xcolor}
\sethlcolor{pink}
\begin{document}
\title{\bf Stability and Finite-Time Blow-Up for a Fractionally Damped Nonlinear Plate Equation: Numerical and Analytical Insights}
	\author{
		\small{\bf Iqra Kanwal}
        \footnote{Corresponding author. Email: ikanwal1703@gmail.com} \\
		\small School of Mathematics and Statistics, Shanxi University, Taiyuan, China \\
		\small {\bf Jianghao Hao}  \footnote{ Email: hjhao@sxu.edu.cn }\\
		\small School of Mathematics and Statistics, Shanxi University, Taiyuan, China,\\
        \small{\bf Muhammad Fahim Aslam}
        \footnote {Email: fahim.sihaab@gmail.com}\\
		\small Department of Mathematics, University of Kotli Azad Jammu and Kashmir(UOKAJK),Kotli, Pakistan\\
         \small{\bf Mauricio Sepúlveda-Cortés}
        \footnote {Email: maursepu@udec.cl}\\
		\small DIM and CI$^2$MA, 
        Universidad de Concepci\'on, Concepci\'on, Chile
       }
	\date{}
	\date{}
	\maketitle
	\begin{abstract}
This paper studies a nonlinear plate equation with internal fractional damping and a time-delay term, driven by a polynomial-type nonlinear source. Such a model arises naturally in the description of viscoelastic and feedback-controlled elastic structures. We first establish the local existence and uniqueness of weak solutions using semigroup theory. The long-time behavior of solutions is then analyzed by constructing a suitable Lyapunov functional, from which stability and energy decay results are obtained. Moreover, by applying the concavity method, we prove that solutions associated with negative initial energy blow up in finite time. These results highlight the competing effects of fractional damping and delayed feedback on the qualitative behavior of the system. Finally, numerical simulations are presented to confirm the analytical results and to illustrate both stability and blow-up dynamics.
\end{abstract}
\section{Introduction}
The study of higher-order partial differential equations (PDEs) plays a central role in mathematical physics, with applications in elasticity, engineering, vibrations, nuclear science, and material modeling. Among them, the plate equation is a basic model that describes the oscillations and bending of thin elastic structures. The qualitative behavior of solutions to plate equations depends strongly on the interplay of damping mechanisms, nonlinear source terms, and time delays, which can lead to stability, general decay, or blow-up in finite time. 
In this paper, we investigate the following nonlinear plate equation with internal fractional damping and a delayed feedback term:
\begin{equation}\label{1.1}
\begin{cases}
\mathcal{V}_{tt} + \Delta^2 \mathcal{V}+\partial_{t}^{\theta,\vartheta} \mathcal{V}(x,t)+a_1\mathcal{V}_{t}+a_2\mathcal{V}_{t}(x,t-s) = \mathcal{V}|\mathcal{V}|^{p-2}, & (x,t) \in \Omega \times (0,\infty), \\
\mathcal{V}(x,0) = \mathcal{V}_{0}(x), \quad \mathcal{V}_{t}(x,0) = \mathcal{V}_{1}(x), & x \in \Omega, \\
\mathcal{V}=\dfrac{\partial \mathcal{V}}{\partial \nu}=0, & (x,t) \in \partial\Omega \times (0,\infty), \\
\mathcal{V}_{t}(x,t-s) =f_{0}(x,t-s), & x \in \Omega, \; t \in(0,s),
\end{cases}
\end{equation}
where $\Omega \subset \mathbb{R}^n$ is a bounded domain having a smooth boundary. $\partial \Omega$, $a_1, a_2 > 0$ are constants. We investigate two distinct regimes: assumption (A1), leading to stability and decay, and assumption (A2), associated with delay-dominated dynamics and blow-up, where as  $p>2$ characterizes the power of the nonlinear source term. The operator $\partial_{t}^{\theta,\vartheta}$ with $0<\theta<1$ and $\vartheta>0$ denotes the generalized Caputo fractional derivative \cite{BCA,CM}:  
\[
\partial_{t}^{\theta, \vartheta} \mathcal{V}(t)=\frac{1}{\Gamma(1-\theta)} \int_{0}^{t}(t-s)^{-\theta} e^{-\vartheta(t-s)} \mathcal{V}_{s}(s)\, d s.
\]
\medskip
\noindent
\textbf{Motivation.}  
The nonlinear source term $ \mathcal{V}|\mathcal{V}|^{p-2}$ ($p>2$) is a classical power-type nonlinearity that frequently appears in fluid mechanics, nonlinear acoustics, structural vibrations, and elasticity. It is well known that such nonlinearities can act as destabilizing forces, leading to finite-time blow-up when $E(0)<0$ \cite{BJM}.    
In the present work, we consider a polynomial source term of the form $\mathcal{V}|\mathcal{V}|^{p-2}$ with $p>2$ instead of logarithmic nonlinearities. There has been several researchers considering this type of nonlinearity. For this we recommend \cite{Boulaaras2020,AB,Kirane2025,Kirane2003,Peng2025,Safsaf2024}.There are several motivations for this choice.  

First, power-type nonlinearities are classical in the theory of nonlinear PDEs and arise naturally in diverse physical models such as nonlinear elasticity, acoustics, fluid dynamics, and wave propagation. They also provide a flexible framework since the exponent $p$ determines the growth rate of the source term and influences stability and blow-up thresholds.  

Second, polynomial nonlinearities are mathematically tractable, as they align with Sobolev embeddings and variational structures. This makes it possible to apply powerful tools such as energy methods, Sobolev inequalities, and the concavity method. In contrast, logarithmic nonlinearities exhibit weaker growth and are more restrictive in physical applications.  

Finally, the polynomial case represents a stronger destabilizing mechanism compared to the logarithmic one. Analyzing this case allows us to obtain sharper insights into the competition between the dissipative effects of fractional damping and delay feedback versus the destabilizing influence of nonlinear sources. Once the polynomial case is well understood, extensions to logarithmic or other nonlinearities can be pursued as natural generalizations.  

The use of \emph{fractional damping} is motivated by the modeling of viscoelastic and complex materials, where memory and hereditary properties are intrinsic. Classical integer-order damping fails to fully capture such effects, while fractional operators provide a more realistic description \cite{VM,MC,MAG}. For instance, fractional damping has been successfully applied in modeling polymers, biological tissues, heat conduction in heterogeneous media, and anomalous diffusion processes \cite{MAT,AB,BB,BS,FH,AG5,AS,TB,FA1}. In the context of PDEs, fractional damping not only enhances the accuracy of physical models but also fundamentally changes stability and decay behavior.  

The \emph{time-delay term} represents another key physical phenomenon. Delays naturally occur in systems involving feedback, control, and signal transmission. For example, in mechanical structures, delays may result from sensing and actuation lag; in control theory, feedback delays can destabilize otherwise stable systems; and in engineering applications, delays are inevitable in communication or signal processing. Such delays are known to significantly influence solution behavior, possibly generating oscillations, instabilities, or blow-up phenomena \cite{KM,AO,MU,Mk,NP}.  

Thus, the combination of fractional damping, delayed velocity feedback, and nonlinear source terms provides a rich and realistic mathematical framework for modeling physical systems in viscoelasticity, structural engineering, and control theory. At the same time, it raises challenging analytical questions about well-posedness, stability and blow-up dynamics.\\
\medskip
\noindent
\textbf{Related work.}  
Classical plate models with viscoelastic damping, strong damping, or boundary dissipation have been extensively studied, yielding results on global existence, energy decay, and stabilization and blow up \cite{MU,Mk,Lj,MR,kom,MS,MU1,CD,AG4,DA}. More recently, the introduction of fractional damping in wave and plate equations has attracted significant attention, with works proving well-posedness, general decay, and blow-up phenomena \cite{AB,BB,BS}. However, most of these studies have considered either fractional damping or delay terms separately. The simultaneous presence of fractional-type damping, delay feedback and nonlinear sources has not yet been systematically analyzed in the context of plate equations.  

\medskip
\noindent
\textbf{Our contribution.}  
To the best of our knowledge, this is the first work to address equation \eqref{1.1}, where fractional damping and a delayed velocity feedback act together with a nonlinear source term of power type. Our main contributions are:  
\begin{itemize}
    \item Semigroup theory is used to establish the local existence and uniqueness of weak solutions.
    \item \textit{To prove exponential stability, a Lyapunov functional is utilized}.
    \item Proving that solutions with \textit{negative initial energy blow up in finite time}, via a Lyapunov functional and the concavity method.  
    \item Providing \textit{numerical simulations} that capture the exponential stability as well as  blow-up dynamics and confirm the theoretical findings.  
\end{itemize}
\medskip
\noindent
\textbf{Organization.}  
The paper is structured as follows: In Section~2, we reformulate the problem into an augmented system. Section~3 proves local well-posedness. Sections~4 and 5 establish global existence and energy decay results respectively. Section~6 provides blow-up results in the negative energy case. Section~7 illustrates the theoretical results with numerical simulations.
\section{ Preliminaries}
This section focuses on transforming the problem \eqref{1.1} into an augmented system. The following claims are required to accomplish this.
\begin{lemma}\label{lem1}\textnormal{\cite{MBO}}
Let $\beta$ be the function that we will use to represent the damping effect for $\alpha \in \mathbb{R}   \ \textit{on interval} \ 0<\theta<$1 and this function is defined as follows:
$$
\beta(\alpha)=|\alpha|^{\frac{(2 \theta-1)}{2}}.
$$
Thus, the relationship between the 'input' \( U \) and the 'output' \( O \) of the system is described as follows:
\begin{equation}\label{eq2.1}
\left\{\begin{array}{l}
\mathscr{G}_{t}(x, \alpha, t)+\left(\alpha^{2}+\vartheta\right) \mathscr{G}(x, \alpha, t)-U(x, t) \beta(\alpha)=0 \quad \alpha \in \mathbb{R}, \ t>0, \ \vartheta>0\\
\mathscr{G}(x, \alpha, 0)=0 \\
O(t):=(\pi)^{-1} \sin (\theta \pi) \int_{-\infty}^{+\infty} \mathscr{G}(x, \alpha, t) \beta(\alpha) \, d \alpha
\end{array}\right.
\end{equation}
is given by
$$
O:=I^{1-\theta, \vartheta} U,
$$
where
$$
I^{\theta, \vartheta}\mathcal{V}(t):=\frac{1}{\Gamma(\theta)} \int_{0}^{t}(t-s)^{\theta-1} e^{-\vartheta(t-s)} \mathcal{V}(s) d s.
$$
\end{lemma}
\begin{lemma}\label{lem2}\textnormal{\cite{BB}}
If $\left.\lambda \in D_{\vartheta}=\mathbb{C} \backslash\right]-\infty,-\vartheta]$ then
$$
\int_{-\infty}^{+\infty} \frac{\beta^{2}(\alpha)}{\lambda+\vartheta+\alpha^{2}} d \alpha=\frac{\pi}{\sin (\theta \pi)}(\lambda+\vartheta)^{\theta-1}.
$$
\end{lemma}
We assume the following conditions on the damping and delay coefficients:
\begin{equation}
\tag{A1}
a_1 > a_2 + 2 b A_0,
\end{equation}
Now, like in \cite{NP}, we introduce the new variable
\begin{equation}\label{eq2.3}
z(x, \varrho, t)=\mathcal{V}_{t}(x, t-s \varrho), \quad x \in \Omega, \quad \varrho \in(0,1), \quad t>0.
\end{equation}
Then, we have
\begin{equation}\label{eq2.4}
s z_{t}(x, \varrho, t)+z_{\varrho}(x, \varrho, t)=0, \quad x \in \Omega, \quad \varrho \in(0,1), \quad t>0
\end{equation}
Therefore, by \eqref{eq2.3}-\eqref{eq2.4} and using lemma \ref{lem1}, problem \eqref{1.1} is equivalent to
\begin{equation}\label{2..1}
\begin{cases}
\mathcal{V}_{tt} + \Delta^2 \mathcal{V} +b\int_{-\infty}^{+\infty}\mathscr{G}(x,\alpha,t)\beta(\alpha)d\alpha+a_1\mathcal{V}_{t}+a_2z(x,1,t)= \mathcal{V}|\mathcal{V}|^{p-2} & \text{in } \Omega \times (0,\infty), \\
\mathscr{G}_{t}(x, \alpha, t)+\left(\alpha^{2}+\vartheta\right) \mathscr{G}(x, \alpha, t)-z(x, 1, t) \beta(\alpha)=0 & x \in \Omega, \alpha \in \mathbb{R}, \mathrm{t}>0, \\ 
s z_{t}(x, \varrho, t)+z_{\varrho}(x, \varrho, t)=0 & x \in \Omega, \varrho \in(0,1), \mathrm{t}>0, \\ 
\mathcal{V}=\frac{\partial \mathcal{V}}{\partial \nu}=0 & \text{on } \partial\Omega \times (0,\infty), \\ 
z(x, 0, t)=\mathcal{V}_{t}(x, t) & x \in \Omega, t>0, \\
\mathcal{V}(x, 0)=\mathcal{V}_{0}(x), \quad \mathcal{V}_{t}(x, 0)=\mathcal{V}_{1}(x) & x \in \Omega, \\ 
\mathscr{G}(x, \alpha, 0)=0 & x \in \Omega, \alpha \in \mathbb{R}, \\
z(x, \varrho, 0)=f_{0}(x,-\varrho s) & x \in \Omega, \varrho \in(0,1),
\end{cases}
\end{equation}
where $b:=(\pi)^{-1} \sin (\theta \pi)$.\\
\begin{lemma}\label{lem3}
For $z \in L^{2}(\Omega)$ and $\alpha \mathscr{G} \in L^{2}(\Omega \times(-\infty,+\infty))$, we have
$$
\begin{aligned}
\left|\int_{\Omega} z(x, \varrho, t) \int_{-\infty}^{+\infty} \beta(\alpha) \mathscr{G}(x, \alpha, t) d \alpha d x\right| &\leq  A_{0} \int_{\Omega}|z(x, \varrho, t)|^{2} d x \\
& +\frac{1}{4} \int_{\Omega} \int_{-\infty}^{+\infty}\left(\alpha^{2}+\vartheta\right)|\mathscr{G}(x, \alpha, t)|^{2} d \alpha d x
\end{aligned}
$$
 where $A_{0}$ is a positive constant and can be obtained by definition of Lemma \ref{lem2}.
\end{lemma}
\begin{proof}
Using the Cauchy-Schwarz inequality, we obtain
$$
\left|\int_{-\infty}^{+\infty} \beta(\alpha) \mathscr{G}(x, \alpha, t) d \alpha\right| \leq\left(\int_{-\infty}^{+\infty} \frac{\beta^{2}(\alpha)}{\alpha^{2}+\vartheta} d \alpha\right)^{\frac{1}{2}}\left(\int_{-\infty}^{+\infty}\left(\alpha^{2}+\vartheta\right)|\mathscr{G}(x, \alpha, t)|^{2} d \alpha\right)^{\frac{1}{2}}.
$$
with Young's inequality, we obtain
$$
\begin{aligned}
\left|\int_{\Omega} z(x, \varrho, t) \int_{-\infty}^{+\infty} \beta(\alpha) \mathscr{G}(x, \alpha, t) d \alpha d x\right| &\leq  A_{0} \int_{\Omega}|z(x, \varrho, t)|^{2} d x \\
& +\frac{1}{4} \int_{\Omega} \int_{-\infty}^{+\infty}\left(\alpha^{2}+\vartheta\right)|\mathscr{G}(x, \alpha, t)|^{2} d \alpha d x
\end{aligned}
$$
with
\begin{equation}\label{def.A_0}
A_{0}=\int_{-\infty}^{+\infty} \frac{\beta^{2}(\alpha)}{\alpha^{2}+\vartheta} d \alpha .
\end{equation}
This ends the proof.
\end{proof}
The energy associated with the problem \eqref{2..1} is 
\begin{equation}\label{eq2.5}
\begin{aligned}
E(t)& = \frac{1}{2}\left\|\mathcal{V}_{t}\right\|_{2}^{2}+\frac{b}{2} \int_{\Omega} \int_{-\infty}^{+\infty}|\mathscr{G}(x, \alpha, t)|^{2} d \alpha d x+\frac{1}{2}\|\Delta \mathcal{V}\|_{2}^{2} \\
 &-\frac{1}{p}\|\mathcal{V}\|_{p}^{p}+v s \int_{\Omega} \int_{0}^{1}|z(x, \varrho, t)|^{2} d \varrho d x, 
\end{aligned}
\end{equation}
where $v$ is a positive constant that verifies
\begin{equation}\label{eq2.6}
\frac{a_{2}}{2}+b A_{0}<v<a_1-b A_{0}-\frac{a_{2}}{2},
\end{equation}
which is possible provided \textbf{(A1)}.
\begin{lemma}\label{lem4}
Assume that condition \eqref{eq2.6} is feasible. For the exponent $p$, one requires
\begin{equation}\label{eq2.7}
    \begin{cases}
        2 < p < \infty, & \text{when } n=1 \text{ or } n=2, \\[6pt]
        2 < p \leq \dfrac{2n}{n-2}, & \text{when } n \geq 3 .
    \end{cases}
\end{equation}
Then energy 
$E(t)$ in \eqref{eq2.5} satisfies the decay estimate
\begin{equation}\label{eq2.8}
\begin{aligned}
\frac{d E(t)}{d t} \leq & -C \int_{\Omega} \left( |z(x, 0, t)|^{2} + |z(x, 1, t)|^{2} \right) dx - \frac{b}{2} \int_{\Omega} \int_{-\infty}^{+\infty} \left( \vartheta + \alpha^{2} \right) |\mathscr{G}(x, \alpha, t)|^{2} , d\alpha  dx,
\end{aligned}
\end{equation}
where 
C is a positive constant.
\end{lemma}
\begin{proof}
The first equation of \eqref{2..1} is multiplied by $\mathcal{V}_{t}$, integrated over $\Omega$, and then integrated by parts, yields
\begin{equation}\label{eq2.9}
\begin{aligned}
 \frac{d}{d t}&\left\{\frac{1}{2}\left\|\mathcal{V}_{t}\right\|_{2}^{2}+\frac{1}{2}\|\Delta \mathcal{V}\|_{2}^{2}-\frac{1}{p}\|\mathcal{V}\|_{p}^{p}\right\}\\ &+a_{1}\left\|\mathcal{V}_{t}\right\|_{2}^{2}
+a_{2}\int_{\Omega}\mathcal{V}_{t}z(x,1,t) \\ & +b \int_{\Omega} \mathcal{V}_{t} \int_{-\infty}^{+\infty} \beta(\alpha) \mathscr{G}(x, \alpha, t) d \alpha d x=0.
\end{aligned}
\end{equation}
When we multiply  \eqref{2..1}$_2$ by $b \mathscr{G}$ and integrate over $\Omega \times(-\infty,+\infty)$, we get:
\begin{equation}\label{eq2.10}
\begin{aligned}
 \frac{d}{d t}&\left\{\frac{b}{2} \int_{\Omega} \int_{-\infty}^{+\infty} |\mathscr{G}(x, \alpha, t)|^{2} d \alpha d x\right\} \\
&-b \int_{\Omega} z(x, 1, t) \int_{-\infty}^{+\infty} \beta(\alpha) \mathscr{G}(x, \alpha, t) d \alpha d x
\\
&+b \int_{\Omega} \int_{-\infty}^{+\infty}\left(\alpha^{2}+\vartheta\right)|\mathscr{G}(x, \alpha, t)|^{2} d \alpha d x=0.
\end{aligned}
\end{equation}
Multiply  \eqref{2..1}$_3$ by $2 v z$ and integrate over $\Omega \times(0,1)$, to obtain:
\begin{equation}\label{eq2.11}
\begin{aligned}
&\frac{d}{d t}\left\{s v \int_{\Omega} \int_{0}^{1}|z(x, \varrho, t)|^{2} d \varrho d x\right\}
\\ &+v \int_{\Omega}\left[|z(x, 1, t)|^{2}-|z(x, 0, t)|^{2}\right] d x=0.
\end{aligned}
\end{equation}
By adding up \eqref{eq2.9}-\eqref{eq2.11}, using Poincaré and Young's inequality, and using $\mathcal{V}_{t}=z(x, 0, t)$, we get
\begin{equation}\label{eq2.12}
\begin{aligned}
\frac{d E(t)}{d t}\leq&-(v-bA_0-\frac{a_2}{2})\int_{\Omega}|z(x, 1, t)|^{2} d x \\
& -(b-\frac{b}{4}-\frac{b}{4})\int_{\Omega} \int_{-\infty}^{+\infty}\left(\alpha^{2}+\vartheta\right)|\mathscr{G}(x, \alpha, t)|^{2} d \alpha d x\\ & -\left(a_{1}-bA_0-\frac{a_2}{2}-v\right) \int_{\Omega}|z(x, 0, t)|^{2} d x
\end{aligned}
\end{equation}
Under assumption \textbf{(A1)}, condition \eqref{eq2.6} is feasible.
Hence, choosing $v$ accordingly, both coefficients in \eqref{eq2.12}
are strictly positive, which guarantees the existence of a constant
$C>0$ defined in \eqref{eq2.13}.\\
By applying Lemma \ref{lem3}, the following result will be obtained.
\begin{equation}\label{eq13}
\begin{aligned}
\frac{d E(t)}{d t} \leq & -C \int_{\Omega}\left(|z(x, 1, t)|^{2}+|z(x, 0, t)|^{2}\right) d x \\
& -\frac{b}{2} \int_{\Omega} \int_{-\infty}^{+\infty}\left(\alpha^{2}+\vartheta\right)|\mathscr{G}(x, \alpha, t)|^{2} d \alpha d x,
\end{aligned}
\end{equation}
with
\begin{equation}\label{eq2.13}
C=\min \left\{\left(v-b A_{0}-\frac{a_2}{2}\right),\left(a_1-\frac{a_2}{2}-b A_{0}-v\right)\right\}.
\end{equation}
The constant C is positive since $v$ is selected in a way that satisfies assumption \eqref{eq2.6}.  The proof is complete.
\end{proof}
\section{Well-posedness}
By introducing the notation $y=\mathcal{V}_{t}$ and setting 
\[
U=\begin{pmatrix}\mathcal{V}\\ y\\ \mathscr{G}\\ z\end{pmatrix},
\qquad   
J(U)=\begin{pmatrix}0\\ |\mathcal{V}|^{p-2}\mathcal{V}\\ 0\\ 0\end{pmatrix},
\]
the system in \eqref{2..1} can be reformulated in the abstract form
\begin{equation}\label{eq3.1}
\begin{cases}
\dfrac{d}{dt}U(t)+AU(t)=J(U(t)), \\[6pt]
U(0)=\begin{pmatrix}\mathcal{V}_{0}\\ \mathcal{V}_{1}\\ 0\\ f_{0}(-\varrho s)\end{pmatrix}.
\end{cases}
\end{equation}
Here, the operator $A$ is specified in the following manner:
$$
A U=\left(\begin{array}{l}
-y \\
\Delta^2 \mathcal{V}+b \int_{-\infty}^{+\infty} \mathscr{G}(x, \alpha) \beta(\alpha) d \alpha+a_{1}y +a_2z(1,.)\\
\left(\alpha^{2}+\vartheta\right) \mathscr{G}(x, \alpha)-z(x, 1) \beta(\alpha) \\
\frac{1}{s} z_{\varrho}(x, \varrho)
\end{array}\right),
$$
with domain
\[
D(A)=\Big\{\, U\in \mathcal{H}\;:\;
\mathcal{V}\in H^{2}(\Omega),\; y\in H_{0}^{1}(\Omega),\;
z_{\varrho}\in L^{2}(\Omega\times(0,1)),\;
y=z(\cdot,0), 
\]
\[
\alpha\,\mathscr{G}\in L^{2}(\Omega\times\mathbb{R}),\;
(\alpha^{2}+\vartheta)\mathscr{G}-\beta(\alpha)\,z(x,1,t)\in L^{2}(\Omega\times\mathbb{R})
\,\Big\}.
\]
where the space $\mathcal{H}$ is defined by:
$$
\mathcal{H}:=H_{0}^{1}(\Omega) \times L^{2}(\Omega) \times L^{2}(\Omega \times(-\infty,+\infty)) \times L^{2}(\Omega \times(0,1)),
$$
equipped with the inner product
\[
\begin{aligned}
\langle U , \bar{U} \rangle_{\mathcal{H}}
&= \int_{\Omega} \Big( \Delta \mathcal{V}\,\Delta \bar{\mathcal{V}} \;+\; y\,\bar{y} \Big)\,dx \\[4pt]
& +\, b \int_{\Omega} \int_{\mathbb{R}} \mathscr{G}(x,\alpha)\,\overline{\mathscr{G}(x,\alpha)}\, d\alpha\,dx \\[4pt]
& +\, 2\nu s \int_{0}^{1} \int_{\Omega} z(x,\varrho)\,\bar{z}(x,\varrho)\, dx\, d\varrho .
\end{aligned}
\]
\begin{theorem}
Assume \textbf{(A1)} and \eqref{eq2.7}. For any $U_{0} \in \mathcal{H}$, the problem \eqref{eq3.1} has a local unique weak solution.
$$
U \in C([0, T), \mathcal{H}) .
$$
\end{theorem}
\begin{proof}
Our aim is to verify that $J$ is locally Lipschitz continuous and that $A$ is a maximal monotone operator, cf. \cite{PZ}.  
From \eqref{eq3.1} and \eqref{eq2.8}, one obtains for every $U \in D(A)$
\begin{equation}
\begin{aligned}
\langle AU , U \rangle_{\mathcal{H}}
&\;\geq\; \frac{b}{2}\int_{\Omega}\!\int_{\mathbb{R}}
\big(\vartheta+|\alpha|^{2}\big)\,|\mathscr{G}(x,\alpha)|^{2}\,d\alpha\,dx  \\[4pt]
&\quad +\, C \int_{\Omega} \Big( |z(x,0)|^{2}+|z(x,1)|^{2} \Big)\,dx .
\end{aligned}
\end{equation}
To conclude maximality, it suffices to establish that the mapping $I+A$ is onto.  
That is, for any element $F=(f_{1},f_{2},f_{3},f_{4})^{T}\in \mathcal{H}$, we must find a vector  
$U=(\mathcal{V},y,\mathscr{G},z)^{T}\in D(A)$ solving
\[
(I+A)U = F .
\]
Then,
\begin{equation}\label{eq3.2}
\left\{\begin{array}{l}
\mathcal{V}-y=f_{1}(x),\\
y+\Delta^2 \mathcal{V}+b \int_{-\infty}^{+\infty} \mathscr{G}(x, \alpha) \beta(\alpha) d \alpha+a_{1}y+a_{2}z(1,.)=f_{2}(x), \\
\mathscr{G}(x, \alpha)+\left(\alpha^{2}+\vartheta\right) \mathscr{G}(x, \alpha)-z(x, 1) \beta(\alpha)=f_{3}(x, \alpha), \\
z(x, \varrho)+\frac{1}{s} z_{\varrho}(x, \varrho)=f_{4}(x, \varrho).
\end{array}\right.
\end{equation}
Assume $\mathcal{V}$ has been obtained with sufficient regularity. Therefore, \eqref{eq3.2}$_{1}$ and \eqref{eq3.2}$_{3}$ gives
\begin{equation}\label{eq3.3}
y=\mathcal{V}-f_{1}
\end{equation}
and
\begin{equation}\label{eq3.4}
\mathscr{G}(x, \alpha)=\frac{f_{3}(x, \alpha)+z(x, 1) \beta(\alpha)}{\alpha^{2}+\vartheta+1} \quad x \in \Omega, \alpha \in \mathbb{R}.
\end{equation}
On the other hand, condition \eqref{eq3.2}$_{4}$ together with 
$z(x,0)=\mathcal{V}-f_{1}$ ensures the existence of a unique solution.  
Consequently, we obtain
\begin{equation}\label{eq3.5}
z(x,\varrho) = \big(\mathcal{V}-f_{1}(x)\big)\,e^{-s\varrho}
+ s\,e^{-s\varrho}\int_{0}^{\varrho} e^{s\tau}\,f_{4}(x,\tau)\,ds,
\qquad x\in\Omega,\;\varrho\in(0,1).
\end{equation}
When we substitute \eqref{eq3.3} and \eqref{eq3.4} into \eqref{eq3.2}$_{2}$, we get
\begin{equation}\label{eq3.6}
\begin{aligned}
& \left(1+a_{1}\right)\mathcal{V} + \Delta^2\mathcal{V}+ a_{2} z(1,.)  \\
& +b \int_{-\infty}^{+\infty} \frac{f_{3}(x, \alpha) + z(x, 1) \beta(\alpha)}{\alpha^{2} + \vartheta + 1} \beta(\alpha) d \alpha 
 \\ & =  f_{2}(x) + \left(1 + a_{1}\right) f_{1}(x).
\end{aligned}
\end{equation}
By using \eqref{eq3.6} and \eqref{eq3.5}, we have
\begin{equation}
\begin{aligned}
(1 +a_1+ b A_1 e^{-s \varrho}+a_2e^{-s}) \mathcal{V} & + \Delta^2 \mathcal{V} = (1+a_1+a_2e^{-s}+ b A_1 e^{-s \varrho})f_1+f_2(x)\nonumber \\
& \quad \quad \quad- b A_1 s e^{-s \varrho} \int_0^{\varrho} e^{s \tau} f_4(x, t)-a_2 s e^{-s} \int_0^{1} e^{s \tau} f_4(x, \tau)\nonumber \\
& \quad \quad \quad +b \int_{-\infty}^{+\infty}\frac{f_{3}(x, \alpha) \beta(\alpha)}{\alpha^{2}+\vartheta+1}.
\end{aligned}
\end{equation}
\begin{equation}\label{eq3.7}
\varrho_{*} \mathcal{V} + \Delta^2 \mathcal{V} = G,
\end{equation}
where
\begin{align*}
\varrho_{*} & = 1 +a_1+ b A_1 e^{-s \varrho}+a_2e^{-s},\\
A_1 & = \int_{-\infty}^{+\infty} \frac{\beta^2(\alpha)}{\alpha^2 + \vartheta + 1} \, d\alpha, \\
G & = (1+a_1+a_2e^{-s}+ b A_1 e^{-s \varrho})f_1+f_2(x)\\
& \quad- b A_1 s e^{-s \varrho} \int_0^{\varrho} e^{s \tau} f_4(x, \tau)-a_2 s e^{-s} \int_0^{1} e^{s \tau} f_4(x, \tau) \\
& \quad+b \int_{-\infty}^{+\infty}\frac{f_{3}(x, \alpha) \beta(\alpha)}{\alpha^{2}+\vartheta+1}.
\end{align*}
As a consequence, equation \eqref{eq3.7} can be reformulated in the variational form
\begin{equation}\label{eq3.8}
B(\mathcal{V},w) = L(w), \qquad \forall\, w \in H_{0}^{1}(\Omega).
\end{equation}
Here, the bilinear mapping 
\[
B : H_{0}^{1}(\Omega)\times H_{0}^{1}(\Omega) \to \mathbb{R}
\]
is given by
\begin{equation}\label{eq3.9}
B(\mathcal{V},w) = \varrho_{*}\int_{\Omega}\mathcal{V}\,w\,dx
+ \bar{\lambda}\int_{\Omega} \Delta \mathcal{V}\,\Delta w\,dx ,
\end{equation}
while the linear functional 
\[
L : H_{0}^{1}(\Omega) \to \mathbb{R}
\]
is defined through
\begin{equation}\label{eq3.10}
L(w) = \int_{\Omega} G\,w\,dx .
\end{equation}
The coercive and continuous natures of $B$ and $L$ are straightforward to demonstrate.  Using the Lax-Milgram theorem, we can deduce that \eqref{eq3.8} has a unique solution $\mathcal{W} \in H_{0}^{1}(\Omega)$ for each $w \in H_{0}^{1}(\Omega)$.  By using the classical elliptic regularity, $\mathcal{W} \in H_{0}^{2}(\Omega)$ can be inferred from \eqref{eq3.8}. 
 Thus, $I+A$ is a surjective operator.\\
Finally, we show that the function $J: \mathcal{H} \rightarrow \mathcal{H}$ is Lipschitz locally. The results from \cite{Kirane2025} make it simple to verify that
\begin{align}
\left \Vert \mathscr{J}(U)-\mathscr{J}(\bar{U})\right \Vert _{H}^{2} &= \left \Vert 0, \mathcal{V}|\mathcal{V}|^{p-2}-\bar{\mathcal{V}}|\bar{\mathcal{V}}|^{p-2},0,0\right \Vert _{H}^{2}\nonumber \\
&= \left \Vert \mathcal{V}|\mathcal{V}|^{p-2}-\bar{\mathcal{V}}|\bar{\mathcal{V}}|^{p-2}\right \Vert _{2}^{2}\nonumber \\
&\leq C \left \Vert \mathcal{V}-\bar{\mathcal{V}}\right \Vert _{H_{0}^{1}(\Omega )}^{2}.
\end{align}
Thus, $\mathscr{J}$ is locally Lipschitz.
\end{proof}
\section{Global existence}\label{S4}
Here, our goal is to demonstrate that there is a solution to the given problem \eqref{2..1} that is globally exist.  We shall first define the functionals below:
\begin{align}
I(t)=&\|\Delta \mathcal{V}(t)\|_{2}^{2}+b \int_{\Omega} \int_{-\infty}^{+\infty}|\mathscr{G}(x,\alpha, t)|^{2} \mathrm{~d} \alpha \mathrm{dx}-\|\mathcal{V}\|_{p}^{p}\nonumber\label{equat4.1}\\
&+v s \int_{\Omega} \int_{0}^{1}|z(x, \varrho, t)|^{2} dx d \varrho,\\
J(t)=&\frac{1}{2}\|\Delta \mathcal{V}(t)\|_{2}^{2}-\frac{1}{p} || \mathcal{V}(t)||_p^p +\frac{b}{2} \int_{\Omega} \int_{-\infty}^{+\infty}|\mathscr{G}(x,\alpha, t)|^{2} \mathrm{~d} \alpha \mathrm{dx}\nonumber \label{equat4.2}\\ 
&+v s \int_{\Omega} \int_{0}^{1}||z(x, \varrho, t)|^{2} dx d \varrho.
\end{align}
We have
\begin{equation}\label{equati4.3}
E(t)=\frac{1}{2}\left\|\mathcal{V}_{t}(t)\right\|_{2}^{2}+J(t).
\end{equation}
\begin{lemma}\label{lem4.1}  For any $U_{0} \in \mathcal{H}$, $I(t)>0,\forall\; t>0$, if following conditions are satisfied:
\begin{equation}\label{eq4.1}
\left\{\begin{array}{l}
C^{p}_{*}\left(\frac{2 p}{(p-2)} E(0)\right)^{\frac{p-2}{2}}<1, \\
I(0)>0,
\end{array}\right.
\end{equation}
\end{lemma}
\begin{proof}
Given the continuity of $\mathcal{V}$ and the restriction $I(0)>0$, there exists a $T^*<T$ such that $I(t)\geq 0,\;\forall\; t\in [0, T^*]$.
In addition, we have
\begin{equation*}
\begin{split}
    J(t) = & \biggr(\frac{p-2}{2p}\biggr)\|\Delta \mathcal{V}(t)\|_{2}^{2} + b \biggr(\frac{p-2}{2p}\biggr) \int_{\Omega} \int_{-\infty}^{+\infty} |\mathscr{G}(x,\alpha, t)|^{2} \, d\alpha \, dx - \frac{1}{p}\|\mathcal{V}(t)\|_{p}^{p}\\
    & +\frac{(p-1) v s}{p} \int_{\Omega} \int_{0}^{1}|z(x, \varrho, t)|^{2} dx d \varrho + \frac{1}{p} I(t).
\end{split}
\end{equation*}
Therefore,
\begin{equation}\label{eq4.2}
\|{\color{white}{.}}\Delta \mathcal{V}(t) \|_{2}^{2} \leq \frac{2 p}{p-2}{\color{white}{.}}J(t)\leq \frac{2 p}{p-2} E(t) \leq \frac{2 p}{p-2}{\color{white}{.}}E(0).
\end{equation}
Therefore, by Sobolev -Poincare inequality and \eqref{eq4.1}, we obtain
\begin{align}\label{eq.4.6}
    \|\mathcal{V}\|_{p}^{p} \leq& C^{p}_{*}\|\Delta \mathcal{V}\|_2^{p}\nonumber\\
    \leq& C^{p}_{*}\left(\frac{2 p}{(p-2)} E(0)\right)^{\frac{p-2}{2}}\left(\|\Delta \mathcal{V}\|_{2}^{2}\right)< \|\Delta \mathcal{V}\|_{2}^{2}.
    \end{align}
From \eqref{eq.4.6}, we can say that $I(t)>0$, \quad  $\forall\;t\in\left[0, T^{*}\right]$.
By repeating this procedure and using the fact that 
\begin{align*}
   \lim_{t \to T^*} C_*^{\,p}
\left(
\frac{2p}{p-2}\, E(0)
\right)^{\frac{p-2}{2}}
< 1,
\end{align*}
we can take $T^*=T.$
\end{proof}
Furthermore, we will prove our result for global existence.
\begin{theorem}
Suppose that \eqref{eq2.7} holds. Then for any $U_{0} \in \mathcal{H}$ satisfying \eqref{eq4.1}, solution of system \eqref{2..1} in bounded and global.
\end{theorem}
\begin{proof}
From \eqref{equati4.3}, we have
\begin{equation*}
\frac{1}{2}\left\|\mathcal{V}_{t}(t)\right\|_{2}^{2}+J(t)=E(t) \leq E(0)
\end{equation*}
$$E(t)\geq \frac{1}{2}\left\|\mathcal{V}_{t}(t)\right\|_{2}^{2}+ \frac{(p-2)}{2p}\|\Delta \mathcal{V}\|_{2}^{2}+\frac{1}{p}I(t),$$\
Because $I(t)>0,$ therefore
$\left\|\mathcal{V}_{t}(t)\right\|_{2}^{2}+\|\Delta \mathcal{V}\|_{2}^{2}\leq RE(0),$\
where $R=\max\left\{2,\frac{2p}{p-2}\right\}$.
It shows that system \eqref{2..1} has a bounded as well as global solution.
\end{proof}
\section{Exponential Stability}\label{S5}
For the time being, we concentrate on the energy decay estimates for the problem \eqref{2..1}.  A perturbed modified energy is defined by $N>0$ and $\epsilon_{1}>0$,
\begin{equation}\label{eq5.1}
B(t):=N E(t)+\epsilon_{1} K_{1}(t)+\epsilon_{1} K_{2}(t)+K_{3}(t), 
\end{equation}
where
\begin{equation}\label{eq5.2}
K_{1}(t):=\int_{\Omega} \mathcal{V}_{t} \mathcal{V} d x+\frac{a_2}{2}\int_\Omega|\mathcal{V}|^{2}dx,
\end{equation}
\begin{equation}\label{eq5.3}
K_{2}(t):=\frac{b}{2} \int_{\Omega} \int_{-\infty}^{+\infty}\left(\alpha^{2}+\vartheta\right)|M(x, \alpha, t)|^{2} d \alpha d x,
\end{equation}
\begin{equation}\label{eq5.4}
K_{3}(t):=s \int_{\Omega} \int_{0}^{1} e^{-s \varrho}|z(x, \varrho, t)|^{2} d \varrho d x,
\end{equation}
where
$$
M(x, \alpha, t):=\int_{0}^{t} \mathscr{G}(x, \alpha, t) dt-\frac{s \beta(\alpha)}{\alpha^{2}+\vartheta} \int_{0}^{1} f_{0}(x,-\varrho s) d \varrho+\frac{\mathcal{V}_{0}(x) \beta(\alpha)}{\alpha^{2}+\vartheta} .
$$
\begin{lemma}\label{lem5.1}
Let $(\mathcal{V},\mathcal{V}_t)$ be regular solution of problem \eqref{2..1}, then we have
\begin{equation}\label{eq5.5}
\begin{aligned}
\int_{\Omega} & \int_{-\infty}^{+\infty}\left(\alpha^{2}+\vartheta\right) \mathscr{G}(x, \alpha, t)  M(x, \alpha, t) d \alpha d x \\
= & \int_{\Omega} \mathcal{V}(x, t) \int_{-\infty}^{+\infty} \mathscr{G}(x, \alpha, t)  \beta(\alpha) d \alpha d x \\
& -s \int_{\Omega} \int_{0}^{1} z(x, \varrho, t) \int_{-\infty}^{+\infty} \beta(\alpha) \mathscr{G}(x, \alpha, t)  d \alpha d \varrho d x\\
& -\int_{\Omega} \int_{-\infty}^{+\infty}|\mathscr{G}(x, \alpha, t) |^{2} d \alpha d x
\end{aligned}
\end{equation}
\end{lemma}
\begin{proof}
Using  \eqref{2..1}$_{2}$, to obtain
$$
\begin{aligned}
\left(\alpha^{2}+\vartheta \right) \mathscr{G}(\alpha, t)= & z(x, 1, t) \beta(\alpha)-\mathscr{G}_{t}( \alpha, t) \\
= & \beta(\alpha)[z(x, 1, t)-z(x, 0, t)] \\
& +\mathcal{V}_{t}(x, t) \beta(\alpha)-\mathscr{G}_{t}(x, \alpha, t) .
\end{aligned}
$$
Observe that
$$
-s \int_{0}^{1} z_{t}(x, \varrho, t) d \varrho=\int_{0}^{1} z_{\varrho}(x, \varrho, t) d \varrho=z(x, 1, t)-z(x, 0, t)
$$
Whereupon
$$
\begin{aligned}
\left(\alpha^{2}+\vartheta\right) \mathscr{G}(x, \alpha, t) = & -s \beta(\alpha) \int_{0}^{1} z_{t}(x, \varrho, t) d \varrho \\
& +\mathcal{V}_{t}(x, t) \beta(\alpha)-\mathscr{G}_{t}(x, \alpha, t)
\end{aligned}
$$
If we integrate the last equality throughout the interval $[0, t]$, we get
$$
\begin{aligned}
\int_{0}^{t}\left(\alpha^{2}+\vartheta\right) \mathscr{G}(x, \alpha, s) d s= & -\operatorname{s} \beta(\alpha) \int_{0}^{1} z(x, \varrho, t) d \varrho \\
& +\operatorname{s} \beta(\alpha) \int_{0}^{1} f_{0}(x,-\varrho s) d \varrho \\
& +\mathcal{V}(x, t) \beta(\alpha)-\mathcal{V}_{0}(x) \beta(\alpha)-\mathscr{G}(\alpha, t)
\end{aligned}
$$
Therefore,
\begin{equation}\label{eq5.6}
\begin{aligned}
\left(\alpha^{2}+\vartheta\right) M( \alpha, t)= & -s \beta(\alpha) \int_{0}^{1} z(x, \varrho, t) d \varrho \\
& +\mathcal{V}(x, t) \beta(\alpha)-\mathscr{G}(x, \alpha, t) 
\end{aligned}
\end{equation}
We get \eqref{eq5.5} by multiplying \eqref{eq5.6} by $\mathscr{G}$ and integrating over $\Omega \times(-\infty,+\infty)$.
\end{proof}
\begin{lemma}\label{lem5.2}
If $(\mathcal{V}, \mathcal{V}_t, \mathscr{G}, z)$ is a solution to the problem \eqref{2..1}, then
\begin{equation}\label{eq5.7}
 \left|K_1(t)\right| \leq \frac{1}{2}\|\mathcal{V}_{t}\|_{2}^{2}+\frac{C_{*}^{2}}{2} \|\mathcal{V}\|_{2}^{2}.
\end{equation}
\end{lemma}
\begin{lemma}\label{lem5.3}
Let $(\mathcal{V}, \mathcal{V}_t, \mathscr{G}, z)$ be regular solution of problem \eqref{2..1}, then we have
\begin{equation}\label{eq5.8}
|K_{3}(t)|\leq Cs \int_{\Omega} \int_{0}^{1} e^{-s \varrho}|z(x, \varrho, t)|^{2} d \varrho d x
\end{equation}
\end{lemma}
\begin{lemma}\label{lem5.4}
Let $(\mathcal{V}, \mathscr{G}, z)$ be the regular solution of problem \eqref{2..1}, then
\begin{equation}\label{eq5.9}
\begin{aligned}
&| K_2(t)| \leq 3 s^{2} A_{0} \int_{\Omega} \int_{0}^{1}|z(x, \varrho, t)|^{2} d \varrho d x+3 A_{0} C_{*}^{2}\| \mathcal{V}\|_{2}^{2}\\
& \quad+\frac{3}{\vartheta} \int_{\Omega} \int_{-\infty}^{+\infty}|\mathscr{G}(x, \alpha, t) |^{2} d \alpha d x
\end{aligned}
\end{equation}
\end{lemma}
\begin{proof}
Invoking \eqref{eq5.6}, to obtain
\begin{equation}\label{eq5.10}
\begin{aligned}
& \left.\left|\int_{\Omega} \int_{-\infty}^{+\infty}\left(\alpha^{2}+\vartheta\right)\right| M(x, \alpha, t)\right|^{2} d \alpha d x \mid \\
& \quad \leq s^{2} A_{0} \int_{\Omega}\left(\int_{0}^{1} z(x, \varrho, t) d \varrho\right)^{2} d x \\
& \quad+A_{0}\|\mathcal{V}\|_{2}^{2}+\int_{\Omega} \int_{-\infty}^{+\infty} \frac{|\mathscr{G}(x, \alpha, t) |^{2}}{\alpha^{2}+\vartheta} d \alpha d x \\
& \quad+2 \int_{\Omega} \int_{-\infty}^{+\infty} \frac{|\mathscr{G}(x, \alpha, t)  \mathcal{V}(x, t) \beta(\alpha)|}{\alpha^{2}+\vartheta} d \alpha d x\\
& \quad+2 s A_{0} \int_{\Omega}\left|\mathcal{V}(x, t) \int_{0}^{1} z(x, \varrho, t) d \varrho\right| d x \\
& \quad+2 s \int_{\Omega} \int_{-\infty}^{+\infty} \frac{\left|\mathscr{G}(x, \alpha, t)  \beta(\alpha) \int_{0}^{1} z(x, \varrho, t) d \varrho\right|}{\alpha^{2}+\vartheta} d \alpha d x
\end{aligned}
\end{equation}
The right hand side of \eqref{eq5.10} will now be estimated.  Hölder's inequality first gives us
\begin{equation}\label{eq5.11}
\int_{0}^{1} z(x, \varrho, t) d \varrho \leq\left(\int_{0}^{1}|z(x, \varrho, t)|^{2} d \varrho\right)^{\frac{1}{2}}
\end{equation}
We apply Young's inequality to the fourth and fifth terms to get
\begin{equation}\label{eq5.12}
\begin{aligned}
\int_{\Omega} \int_{-\infty}^{+\infty} & \frac{|\mathscr{G}(x, \alpha, t)  \mathcal{V}(x, t) \beta(\alpha)|}{\alpha^{2}+\vartheta} d \alpha d x \\
& \leq \frac{A_{0}}{2}\|\mathcal{V}\|_{2}^{2}+\frac{1}{2} \int_{\Omega} \int_{-\infty}^{+\infty} \frac{|\mathscr{G}(x, \alpha, t) |^{2}}{\alpha^{2}+\vartheta} d \alpha d x 
\end{aligned}
\end{equation}
and
\begin{equation}
s \int_{\Omega}\left|\mathcal{V}(x, t) \int_{0}^{1} z(x, \varrho, t) d \varrho\right| d x \leq \frac{s^{2}}{2} \int_{\Omega} \int_{0}^{1}|z(x, \varrho, t)|^{2} d \varrho d x+\frac{1}{2}\|\mathcal{V}\|_{2}^{2}
\end{equation}
We use Lemma \ref{lem3}, Young's inequality, and \eqref{eq5.11}, to the final term to obtain
\begin{equation}\label{eq5.14}
\begin{aligned}
& s \int_{\Omega} \int_{-\infty}^{+\infty} \frac{\left|\mathscr{G}(\alpha, t) \beta(\alpha) \int_{0}^{1} z( \varrho, t) d \varrho\right|}{\alpha^{2}+\vartheta} d \alpha d x \\
& \quad \leq \frac{s^{2} A_{0}}{2} \int_{\Omega} \int_{0}^{1}|z(x, \varrho, t)|^{2} d \varrho d x \\
& \quad+\frac{1}{2} \int_{\Omega} \int_{-\infty}^{+\infty} \frac{|\mathscr{G}(x, \alpha, t) |^{2}}{\alpha^{2}+\vartheta} d \alpha d x
\end{aligned}
\end{equation}
As a result, we find
\begin{equation}\label{eq5.15}
\begin{aligned}
& \left.\left|\int_{\Omega} \int_{-\infty}^{+\infty}\left(\alpha^{2}+\vartheta\right)\right| M(x, \alpha, t)\right|^{2} d \alpha d x \mid \\
&  \leq 3 s^{2} \int_{\Omega} \int_{0}^{1}|z(x, \varrho, t)|^{2} d \varrho d x+3 A_{0}\|\mathcal{V}\|_{2}^{2}\\
& \quad+3 \int_{\Omega} \int_{-\infty}^{+\infty} \frac{|\mathscr{G}(x, \alpha, t) |^{2}}{\alpha^{2}+\vartheta} d \alpha d x
\end{aligned}
\end{equation}
By applying Poincaré's inequality and the fact that $\frac{1}{\alpha^{2}+\vartheta} \leq \frac{1}{\vartheta}$, \eqref{eq5.9} is established.\\
\end{proof}
\begin{lemma}\label{lem5.5}
Assume that \((\mathcal{V}, \mathcal{V}_t, \mathscr{G} )\) is a solution to the problem \eqref{2..1}, satisfying lemma \ref{lem5.2}, lemma \ref{lem5.3}, and lemma \ref{lem5.4}.  Thus,the two positive constants \(\eta_{1}\) and \(\eta_{2}\) exist such that
\begin{equation*}
\eta_{1} \mathrm{E}(\mathrm{t}) \leq \mathrm{B}(\mathrm{t}) \leq \eta_{2} \mathrm{E}(\mathrm{t}). \end{equation*}
\end{lemma}
\begin{proof}
The proof is obvious.
\end{proof}
\begin{lemma}
Let us suppose that \textbf{(A1)} and \eqref{eq2.7} are true. The function $K_{1}$ defined by \eqref{eq5.2} satisfies \begin{equation}\label{eq5.19}.
\begin{aligned}
K_{1}^{\prime}(t)\leq & \left\|\mathcal{V}_{t}\right\|_{2}^{2}+\|\mathcal{V}\|_{p}^{p}-\eta\|\Delta \mathcal{V}\|_{2}^{2}+\frac{a_2}{2}\int_\Omega |z(x,1,t)|^2dx \\
& +\frac{b}{4} \int_{\Omega} \int_{-\infty}^{+\infty}\left(\alpha^{2}+\vartheta\right)|\mathscr{G}(x, \alpha, t) |^{2} d \alpha d x \\
\end{aligned}
\end{equation}
\end{lemma}
Where $ \eta=\frac{a_2}{2}C_{**}^{2}+bA_0C_{*}^{2}-1$
\begin{proof}
A straightforward differentiation of $K_{1}$ and applying \eqref{2..1} provides
\begin{equation}\label{eq5.20}
\begin{aligned}
K_{1}^{\prime}(t)= & \left\|\mathcal{V}_{t}\right\|_{2}^{2}+\int_{\Omega} \mathcal{V}_{t t} \mathcal{V} d x+\frac{a_2}{2}\int_\Omega|\mathcal{V}|^2dx\\
= & \left\|\mathcal{V}_{t}\right\|_{2}^{2}-\|\Delta \mathcal{V}\|_{2}^{2}+\|\mathcal{V}\|_{p}^{p}-a_2\int_\Omega |z(x,1,t)| \mathcal{V}dx \\
& -b \int_{\Omega}\mathcal{V} \int_{-\infty}^{+\infty}\mathscr{G}(x, \alpha, t) \beta(\alpha) d \alpha d x
\end{aligned}
\end{equation}
by using Young's inequality and Lemma \ref{lem3}, we get
$$
\begin{aligned}
K_{1}^{\prime}(t)\leq & \left\|\mathcal{V}_{t}\right\|_{2}^{2}-\|\Delta \mathcal{V}\|_{2}^{2}+\|\mathcal{V}\|_{p}^{p}+\frac{a_2}{2}\|\mathcal{V}\|_{2}^{2}+\frac{a_2}{2}\int_\Omega|z(x,1,t)|^2dx\\
& +bA_0\|\mathcal{V}\|_{2}^{2} +\frac{b}{4} \int_{\Omega} \int_{-\infty}^{+\infty}\left(\alpha^{2}+\vartheta\right)|\mathscr{G}(x, \alpha, t) |^{2} d \alpha d x
\end{aligned}
$$
By using Poincaré's inequality on the $4^{th}$ and $6^{th}$ term of above expression, \eqref{eq5.19} is established.\\
\end{proof}
\begin{lemma}
 Let $(\mathcal{V}, \mathcal{V}_t, \mathscr{G}, z)$ be a solution of \eqref{2..1}, then, functional $K_2$ satisfies,
\begin{equation}\label{eq5.21}
\begin{aligned}
K_{2}^{\prime}(t) \leq & bA_0C_{*}^{2}\|\Delta \mathcal{V}\|_{2}^{2}+\frac{b}{4} \int_{\Omega} \int_{-\infty}^{+\infty}\left(\alpha^{2}+\vartheta \right)|\mathscr{G}(x, \alpha, t) |^{2} d \alpha d x \\
& +sbA_0 \int_{\Omega} \int_{0}^{1}|z(x, \varrho, t)|^{2} d \varrho d x+\frac{bs}{4} \int_{\Omega} \int_{-\infty}^{+\infty}\left(\alpha^{2}+\vartheta \right)|\mathscr{G}(x, \alpha, t) |^{2} d \alpha d x \\
& -b \int_{\Omega} \int_{-\infty}^{+\infty}|\mathscr{G}(x, \alpha, t) |^{2} d \alpha d x
\end{aligned}
\end{equation}
\end{lemma}
\begin{proof}
 By differentiating $K_2$, we obtain
\begin{equation*}
K_{2}^{\prime}(t)=b \int_{\Omega} \int_{-\infty}^{+\infty}\left(\alpha^{2}+\vartheta \right) \mathscr{G}(x,\alpha, t)M(x, \alpha, t) d \alpha d x.
\end{equation*}
Using lemma \ref{lem5.1}, we arrive at
$$
K_{2}^{\prime}(t) = b \int_{\Omega} \int_{-\infty}^{+\infty} \mathscr{G}(x, \alpha, t)  \left[-s \beta(\alpha) \int_{0}^{1}z(x,\varrho,t) d\varrho+\mathcal{V}(x,t)\beta(\alpha)-\mathscr{G}(x,\alpha,t) \right] d\alpha \, dx
$$
Using Young's inequality and lemma \ref{lem3}, we get
$$
\begin{aligned}
K_{2}^{\prime}(t) \leq & bA_0 \left\|\mathcal{V}\right\|_{2}^{2}+\frac{b}{4} \int_{\Omega} \int_{-\infty}^{+\infty}\left(\alpha^{2}+\vartheta \right)|\mathscr{G}(x, \alpha, t) |^{2} d \alpha d x \\
& +sbA_0 \int_{\Omega} \int_{0}^{1}|z(x, \varrho, t)|^{2} d \varrho d x+\frac{bs}{4} \int_{\Omega} \int_{-\infty}^{+\infty}\left(\alpha^{2}+\vartheta \right)|\mathscr{G}(x, \alpha, t) |^{2} d \alpha d x \\
& -b \int_{\Omega} \int_{-\infty}^{+\infty}|\mathscr{G}(x, \alpha, t) |^{2} d \alpha d x
\end{aligned}
$$
By using Poincare's inequality on the first term of above equation, Lemma is proved.
\end{proof}
\begin{lemma}\label{lem5.8}
Assume that \textbf{(A1)} and \eqref{eq2.7} are valid. Thus, the functional $K_{2}$ defined by \eqref{eq5.4} satisfies
\begin{equation}\label{eq5.22}
K_{3}^{\prime}(t) \leq-s e^{-s} \int_{\Omega} \int_{0}^{1}|z(x, \varrho, t)|^{2} d \varrho d x+\left\|\mathcal{V}_{t}\right\|_{2}^{2}
\end{equation}
\end{lemma}
\begin{proof}
Using the $3^{rd}$ equation of \eqref{2..1} and the time derivative of $K_{2}$, we obtain
\begin{align*}
K_{3}^{\prime}(t) & = -2 s \int_{\Omega} \int_{0}^{1} e^{-s \varrho} z(x, \varrho, t) z_{t}(x, \varrho, t) \, d\varrho \, dx \\
& = -2 \int_{\Omega} \int_{0}^{1} e^{-s \varrho} z(x, \varrho, t) z_{\varrho}(x, \varrho, t) \, d\varrho \, dx \\
& = -\int_{\Omega} \int_{0}^{1} \frac{d}{d \varrho}\left[e^{-s \varrho}|z(x, \varrho, t)|^{2}\right] \, d\varrho \, dx - s \int_{\Omega} \int_{0}^{1} e^{-s \varrho}|z(x, \varrho, t)|^{2} \, d\varrho \, dx \\
& = -s \int_{\Omega} \int_{0}^{1} e^{-s \varrho}|z(x, \varrho, t)|^{2} \, d\varrho \, dx - e^{-s} \int_{\Omega}|z(x, 1, t)|^{2} \, dx + \left\|\mathcal{V}_{t}\right\|_{2}^{2}
\end{align*}
Then \eqref{eq5.22} is established.\\
\end{proof}
\begin{theorem}\label{T59}
Let  \textbf{(A1)} and \eqref{eq2.7} hold. Assume that $U_{0} \in \mathcal{H}$ satisfies \eqref{eq4.1} then any solution to \eqref{2..1}  satisfies
\begin{equation}
E(t) \leq K e^{-w t}, \quad t \geq 0, 
\end{equation}
$w$ and $K$ are positive constants that are independent of $t$.
\end{theorem}
\begin{proof}
After differentiating \eqref{eq5.1}, use \eqref{eq13}, \eqref{eq5.19},\eqref{eq5.21}, and \eqref{eq5.22}, we will get
\begin{equation}\label{eq5.24}
\begin{aligned}
B^{\prime}(t) & \leq -(NC-1-\epsilon_1) \left\|\mathcal{V}_t\right\|_{2}^{2} - \epsilon_1(1-bA_0C_{*}^{2})\left\|\Delta\mathcal{V}\right\|_{2}^{2} \\
& \quad - b(\frac{N}{4} - \frac{\epsilon_1 (2+s)}{4} )\int_{\Omega} \int_{-\infty}^{+\infty} \left(\alpha^{2}+\vartheta \right) |\mathscr{G}(x, \alpha, t) |^{2} \, d\alpha \, dx \\
& \quad - s\left( e^{-s}-\epsilon_1 b A_0 \right) \int_{\Omega} \int_{0}^{1} |z(x, \varrho, t)|^{2} \, d\varrho \, dx- \left(NC-\frac{\epsilon_1 a_2}{2}\right) \int_{\Omega} |z(1, t)|^{2} \, dx \\
& \quad - \epsilon_1 b \int_{\Omega} \int_{-\infty}^{+\infty} |\mathscr{G}(x, \alpha, t) |^{2} \, d\alpha \, dx + \epsilon_1|| \mathcal{V}(t)||_p^p
\end{aligned}
\end{equation}
At this stage, we select $N$ is large enough and $\epsilon_{1}$ small enough that
$$
NC-1-\epsilon_1>0, 
\quad \frac{bN}{4} - \frac{\epsilon_1 b(2+s)}{4}>0\\
\quad and \quad
s(e^{-s}-\epsilon_1 b A_0)>0
$$
In light of the previous inequality, we have arrived at the conclusion that there exists a constant $m>0$ convert \eqref{eq5.24} into
\begin{equation}
\begin{aligned}
B^{\prime}(t) & \leq -m \left( \left\|\mathcal{V}_t\right\|_{2}^{2} + \left\|\Delta\mathcal{V}\right\|_{2}^{2}+ \|\mathcal{V}\|_p^p \right.\nonumber \\
& \quad + \int_{\Omega} \int_{-\infty}^{+\infty} |\mathscr{G}(x, \alpha, t) |^{2} \, d\alpha \, dx \left. +\int_{\Omega} \int_{0}^{1} |z(x, \varrho, t)|^{2} \, d\varrho \, dx  \right)
\end{aligned}
\end{equation}
The above expression takes the form
\begin{equation}
B^{\prime}(t) \leq-m E(t), \quad \text { for all } t \geq 0
\end{equation}
By lemma \ref{lem5.5},
\begin{equation}\label{eq5.28}
B^{\prime}(t) \leq-w B(t), \quad \text { for all } t \geq 0 .
\end{equation}
A straightforward integration of \eqref{eq5.28} over $(0, t)$ results in
$$
B(t) \leq B(0) e^{-w t}, \quad t \geq 0 .
$$
Since $B(t)$ and $E(t)$ are equivalent, we get
\begin{equation}
E(t) \leq k e^{-w t}, \quad t \geq 0 .
\end{equation}
\end{proof}
\section{Blow-up Result}
\label{sec6}
In this section, we work under assumption,\\ 
\begin{equation}
\vartheta^{\,1-\theta} < a_2 \tag{A2}
\end{equation}
 which corresponds
to a delay-dominated regime and differs from the stability condition \textbf{(A1)}.
We analyze blow-up behavior of solution to the problem \eqref{2..1} in the following section. For the blow up result, suppose
\begin{equation}\label{eq6.1}
\begin{aligned}
H(t) = -E(t) &=  \frac{1}{p}\|\mathcal{V}\|_{p}^{p}- \frac{1}{2}\left\|\mathcal{V}_{t}\right\|_{2}^{2}- \frac{b}{2} \int_{\Omega} \int_{-\infty}^{+\infty} |\mathscr{G}(x, \alpha, t)|^{2} \, d\alpha \, dx \\
& - \frac{1}{2}\|\Delta \mathcal{V}\|_{2}^{2} - v s \int_{\Omega} \int_{0}^{1} |z(x, \varrho, t)|^{2} \, d\varrho \, dx,
\end{aligned}
\end{equation}
We need the following lemma.
\begin{lemma}[\cite{KM}]\label{lem6.1}
Let $\Omega$ be a bounded domain. Then there exists a constant $C>0$, depending only on $\Omega$, such that for every $\mathcal{V}\in L^{p}(\Omega)$ and for all $s$ with $2 \leq s \leq p$, one has
\begin{equation}
\|\mathcal{V}\|_{p}^{s} \;\leq\; C \Big( \|\Delta \mathcal{V}\|_{2}^{2} \;+\; \|\mathcal{V}\|_{p}^{p} \Big).
\end{equation}
\end{lemma}
\begin{theorem}
\label{thm6.2}
Suppose that conditions \eqref{eq2.7} and (A2) hold, and that the initial energy satisfies
$E(0)<0$. Then the solution of system \eqref{2..1} cannot exist globally in time; in fact, it blows
up after a finite interval.
\end{theorem}
\begin{proof}
From \eqref{eq2.5}, it follows that  
\begin{equation}
E(t) \;\leq\; E(0) \;<\; 0 .
\end{equation}
So,
\begin{equation}\label{eq6.7}
\begin{split}
H'(t) = -E'(t)
& \geq \frac{b}{2} \int_{\Omega} \int_{-\infty}^{+\infty}\left(\alpha^{2}+\vartheta\right)|\mathscr{G}(x, \alpha, t)|^{2} d \alpha d x \geq 0.
\end{split}
\end{equation}
and
\begin{equation}\label{eq6.8}
0 < H(0) \;\leq\; H(t) \;\leq\; \frac{1}{p}\,\|\mathcal{V}(t)\|_{L^{p}(\Omega)}^{p}, 
\qquad t \geq 0 .
\end{equation}
Let
\begin{equation}\label{eq6.9}
A(t) = H^{1-\gamma} (t) + \epsilon \int_\Omega \mathcal{V} \mathcal{V}_t dx+\frac{a_{1} \epsilon}{2}\|\mathcal{V}\|_{2}^{2},
\end{equation}
After that, $\epsilon > 0$ must be stated, and further
The parameter $\gamma$ is restricted to the interval  
\begin{equation}\label{eq52}
\frac{2(p-2)}{p^{2}} \;<\; \gamma \;<\; \frac{p-2}{p} \;<\; 1 .
\end{equation}
Differentiating (\ref{eq6.9})
and using \eqref{2..1}, we get
\begin{equation*}
\begin{aligned}
A'(t) \;=\;&\; \epsilon \,\|\mathcal{V}_{t}\|_{2}^{2} 
- \epsilon \,\|\Delta \mathcal{V}\|_{2}^{2} 
+ (1-\gamma)\,H^{-\gamma}(t)\,H'(t) \\[6pt]
& - \frac{a_{1}\epsilon}{2}\,\|\mathcal{V}\|_{2}^{2} 
- b\epsilon \int_{\Omega} \mathcal{V}(x)\!\int_{\mathbb{R}} \beta(\alpha)\,\mathscr{G}(x,\alpha,t)\,d\alpha\,dx \\[6pt]
& + \epsilon \,\|\mathcal{V}\|_{p}^{p} 
+ \frac{a_{1}\epsilon}{2}\,\|\mathcal{V}\|_{2}^{2} 
- \epsilon a_{2} \int_{\Omega} \mathcal{V}(x)\,z(x,1,t)\,dx .
\end{aligned}
\end{equation*}
\begin{equation}\label{eq6.11}
\begin{aligned}
A'(t) \;=\;&\; \epsilon \,\|\mathcal{V}_{t}\|_{2}^{2} 
- \epsilon \,\|\Delta \mathcal{V}\|_{2}^{2} 
+ (1-\gamma)\,H^{-\gamma}(t)\,H'(t) \\[6pt]
& - b\epsilon \int_{\Omega} \mathcal{V}(x)\!\int_{\mathbb{R}} 
   \beta(\alpha)\,\mathscr{G}(x,\alpha,t)\,d\alpha\,dx \\[6pt]
& + \epsilon \,\|\mathcal{V}\|_{p}^{p} 
- \epsilon a_{2} \int_{\Omega} \mathcal{V}(x)\,z(x,1,t)\,dx .
\end{aligned}
\end{equation}
Apply Young's inequality on last two terms of \eqref{eq6.11} and using \eqref{eq6.7}, we get
\begin{equation}\label{eq:6.9}
\begin{split}
b \int_{\Omega} \mathcal{V} \int_{-\infty}^{+\infty} \beta(\alpha)\,\mathscr{G}(x,\alpha,t)\,d\alpha\,dx
&\le \delta\, bA_{0}\,\|\mathcal{V}\|_{2}^{2}
   +\frac{b}{4\delta}\int_{\Omega}\int_{-\infty}^{+\infty}(\alpha^{2}+\vartheta)\,|\mathscr{G}(x,\alpha,t)|^{2}\,d\alpha\,dx \\
&\le \delta\, bA_{0}\,\|\mathcal{V}\|_{2}^{2}
   +\frac{1}{2\delta}H'(t),
\end{split}
\end{equation}
where $A_{0}$ is the constant defined in \eqref{def.A_0}.\\
For \eqref{def.A_0} and $\delta > 0$, this may change depending on t\\
and\\
\begin{equation}
\epsilon a_2\int_\Omega \mathcal{V} z(x,1,t)\leq \frac{\epsilon a_2}{2}\|\mathcal{V}\|_{2}^{2}+\frac{\epsilon a_2}{2}\int_\Omega|z(x,1, t)|^{2} \
\end{equation}
which yields, by substitution in \eqref{eq6.11},
\begin{equation}\label{eq6.13}
\begin{split}
A'(t)
& \geq \biggl( (1-\gamma) H^{-\gamma}(t) - \frac{\epsilon}{2\delta} \biggl)  H' (t)\\ &+ \epsilon ||\mathcal{V}||_{p}^{p}+ \epsilon ||\mathcal{V}_t||_2^2-\frac{\epsilon a_2}{2}\int_\Omega|z(x,1, t)|^{2}
\\
& - \epsilon \delta A_0 ||\mathcal{V}||_2^2 - \epsilon ||\Delta \mathcal{V}||_2^2 -\frac{\epsilon a_2}{2}\|\mathcal{V}\|_{2}^{2}.
\end{split}
\end{equation}
At this stage, assumption (A2) ensures that the contribution of the delay term
cannot be absorbed by the dissipative terms, which is essential to derive a
positive lower bound for A'(t).\\
Assuming $\delta$ is such that $\frac{1}{2 \delta}=k H^{-\gamma}(t)$, and substituting in \eqref{eq6.13} for $k$ is large enough to be given later, we reach at
\begin{equation}\label{eq6.14}
\begin{split}
A'(t)
& \geq \biggl( (1-\gamma) H^{-\gamma}(t) - \frac{\epsilon}{2\delta} \biggl)  H' (t) \\ &-\frac{A_0 \epsilon}{2 k} H^{\gamma}(t)\|\mathcal{V}\|_{2}^{2}+ \epsilon ||\mathcal{V}||_{p}^{p}-\frac{\epsilon a_2}{2}\|\mathcal{V}\|_{2}^{2}
\\
&+ \epsilon ||\mathcal{V}_t||_2^2 - \epsilon ||\Delta \mathcal{V}||_2^2 -\frac{\epsilon a_2}{2}\int_\Omega|z(x,1, t)|^{2}.
\end{split}
\end{equation}
By using Poincaré's inequality, we reached at
\begin{equation}\label{eq6.14.}
\begin{split}
A'(t)
& \geq \biggl( (1-\gamma) H^{-\gamma}(t) - \frac{\epsilon}{2\delta} \biggl)  H' (t)\\&-\frac{A_0 \epsilon}{2 k} H^{\gamma}(t)\|\mathcal{V}\|_{2}^{2}+\epsilon||\mathcal{V}||_{p}^{p}
\\
& + \epsilon ||\mathcal{V}_t||_2^2 - \epsilon ||\Delta \mathcal{V}||_2^2 \\ &-\frac{\epsilon a_2C_{*}^2}{2}\|\Delta\mathcal{V}\|_{2}^{2}-\frac{\epsilon a_2}{2}\int_\Omega|z(x,1, t)|^{2}dx.
\end{split}
\end{equation}
To insert $\epsilon||\mathcal{V}||_{p}^{p}$ using $H(t)=-E(t)$ for $0 < r < 1$, (\ref{eq6.14.}) becomes:
\\
\begin{equation}
\begin{split}
A'(t)
& \geq \bigr[ (1-\gamma) - \epsilon k \bigr] H^{-\gamma}(t)H'(t) + \epsilon \Bigl( \frac{p(1-r)}{2}+1 \Bigl) ||\mathcal{V}_t||_2^2
\\
&+ \epsilon \Bigr[ \frac{ p(1-r)}{2}-\frac{ a_2 C_{*}^2}{2}-1 \Bigr] ||\Delta \mathcal{V}||_2^2+ \epsilon r ||\mathcal{V}||_p^p 
\\
& + \frac{\epsilon bp(1-r)}{2} \int_\Omega \int_{-\infty}^{+\infty} |\mathscr{G} (x,\alpha,t)|^2 d\alpha dx \\ & + \epsilon {p(1-r)} H(t)
 - \epsilon \frac{A_0}{2k} H^{\gamma}(t) ||\mathcal{V}||_{2}^{2}\\&+\epsilon p(1-r) v s \int_{\Omega} \int_{0}^{1}|z(x, \varrho, t)|^{2} d \varrho d x \\ &
+\frac{\epsilon a_2}{2}\int_{\Omega}|z(x, 1, t)|^{2} d x
\end{split}
\label{eq61}
\end{equation}
Using (\ref{eq52}), the result is,
\begin{equation}
\begin{dcases}
2 < p\gamma + 2 \leq p, \\
2 < \frac{\gamma p^2}{p - 2} \leq p.
\end{dcases}
\end{equation}
By lemma \ref{lem6.1},we have
\begin{equation}\label{eq6.17}
  H^{\gamma}(t)||\mathcal{V}||_2^{2} \leq c\Bigl({||\mathcal{V}||_p^p +||\Delta \mathcal{V}||_2^{2}}\Bigl).
\end{equation}
for some $c A_0=C_2 >$ 0. merging equations (\ref{eq6.17}) and (\ref{eq61}), we get
\begin{equation}
\begin{split}
A'(t)
& \geq \bigr[ (1-\gamma) - \epsilon k \bigr] H^{-\gamma}(t)H'(t) + \epsilon \Bigl( \frac{p(1-r)}{2}+1 \Bigl) ||\mathcal{V}_t||_2^2
\\
&\quad + \epsilon \Bigr[ \frac{ p(1-r)}{2}-\frac{ a_2 C_{*}^2}{2}+\frac{C_2}{2k}-1 \Bigr] ||\Delta \mathcal{V}||_2^2
\\
&\quad + \frac{\epsilon bp(1-r)}{2} \int_\Omega \int_{-\infty}^{+\infty} |\mathscr{G} (x,\alpha,t)|^2 d\alpha dx
\\
&\quad +\epsilon p(1-r) v s \int_{\Omega} \int_{0}^{1}|z(x, \varrho, t)|^{2} d \varrho d x\\& \quad +\frac{\epsilon a_2}{2}\int_{\Omega}|z(x, 1, t)|^{2} d x+ \epsilon {p(1-r)} H(t)\\ & \quad + \epsilon (r+\frac{C_2}{2k}) ||\mathcal{V}||_p^p 
\end{split}
\label{eq64}
\end{equation}
We now choose $r$ to be sufficiently small so that
$$
\frac{ p(1-r)}{2}+1 > 0.
$$
Then, with $r$ fixed, we choose k sufficiently big that
$$
\Bigr[ \frac{ p(1-r)}{2}-\frac{ a_2 C_{*}^2}{2}+\frac{C_2}{2k}-1 \Bigr]> 0,
\quad\Bigr(r+\frac{ C_2}{2k}\Bigr) > 0.
$$
We select a sufficiently small $\epsilon$ while k remains unchanged so that
\begin{align*}
&(1 - \gamma) - \epsilon k > 0, \\
& \hspace{-4.4cm}\text{and} \\
A(0) & = H^{1 - \gamma}(t) + \epsilon \int_{\Omega} \mathcal{V}_0 \mathcal{V}_1 \, dx + \frac{a_{2} \epsilon}{2} \|\mathcal{V}_0\|_{2}^{2} > 0.
\end{align*}
Thus, for some $C_3 >0$, estimate (\ref{eq64}) is
\begin{equation}
\begin{aligned}
\mathcal{A}'(t) \geq C_3 \biggl\{ & H(t) + \|\mathcal{V}_t\|_2^2 + \|\Delta \mathcal{V}\|_2^2 + \|\mathcal{V}\|_p^p + \int_{\Omega} |z(x, 1, t)|^2  dx \\
& + \int_\Omega \int_{-\infty}^{+\infty} |\mathscr{G}(x,\alpha,t)|^2  d\alpha  dx + v s \int_{\Omega} \int_{0}^{1} |z(x, \varrho, t)|^2  d\varrho  dx \biggr\}.
\end{aligned}
\label{eq65}
\end{equation}
and
\begin{equation}\label{GrindEQ__66_}
\mathcal{A}(t) \geq \mathcal{A}(0)>0, \quad t>0.
\end{equation}
On the other hand,we have
\begin{equation}   \label{eq68}
\begin{aligned}
\mathcal{A}^{\frac{1}{1-\gamma}}(t)= & \left[H^{1-\gamma}+\varepsilon \int_{\Omega} \mathcal{V} \mathcal{V}_{t} d x+\frac{a_{2} \epsilon}{2}\int_{\Omega}|\mathcal{V}|_{2} \right]^{\frac{1}{1-\gamma}}\\
    \leq& c \Big[ H(t) + \left| \int_{\Omega} \mathcal{V} \mathcal{V}_{t} \, dx \right|^{\frac{1}{1-\gamma}} + \|\mathcal{V}\|_{2}^{\frac{2}{1-\gamma}} \Big]
    \end{aligned}
\end{equation}
Then, using the embedding theorem and Holder's inequality, we get:
\begin{equation}
\left| \int_{\Omega} \mathcal{V} \mathcal{V}_{t} dx \right| \leq \|\mathcal{V}\|_{2} \cdot \left\|\mathcal{V}_{t}\right\|_{2} \leq C_4 \|\mathcal{V}\|_{p} \cdot \left\|\mathcal{V}_{t}\right\|_{2}.
\label{eq69}
\end{equation}
Young's inequality therefore gives us
\begin{equation}
\begin{aligned}
\left|\int_{\Omega} \mathcal{V} \mathcal{V}_{t} d x\right|^{\frac{1}{1-\gamma}} & \leq c\|\mathcal{V}\|_{p}^{\frac{1}{1-\gamma}} \cdot\left\|\mathcal{V}_{t}\right\|_{2}^{\frac{1}{1-\gamma}} \\
& \leq C_5\left[\|\mathcal{V}\|_{p}^{\frac{\mu}{1-\gamma}}+\left\|\mathcal{V}_{t}\right\|_{2}^{\frac{\theta}{1-\gamma}}\right] .
\end{aligned}
\end{equation}
where $\frac{1}{\theta}+\frac{1}{\theta_1}=1$.
we take $\theta_1=2(1-\gamma)$, to get
$$
\frac{\theta}{1-\gamma}=\frac{2}{1-2\gamma} \leq p.
$$
Therefore, if  $s=2 /(1-2 \gamma)$,then we get:
$$
\left|\int_{\Omega} \mathcal{V} \mathcal{V}_{t} d x\right|^{\frac{1}{1-\gamma}} \leq c\left[\|\mathcal{V}\|_{p}^{s}+\left\|\mathcal{V}_{t}\right\|_{2}^{2}\right].
$$
hence, the lemma \ref{lem6.1} gives
\begin{equation}\
\left|\int_{\Omega} \mathcal{V} \mathcal{V}_{t} d x\right|^{\frac{1}{1-\gamma}} \leq C_5\left[\|\Delta \mathcal{V}\|_{2}^{2}+\left\|\mathcal{V}_{t}\right\|_{2}^{2}+\|\mathcal{V}\|_{p}^{p}\right] .
\label{eq71}
\end{equation}
Putting (\ref{eq71}) in (\ref{eq68}) we have\\
\begin{equation}
\begin{gathered}
\mathcal{A}^{\frac{1}{1-\gamma}}(t) \leq C_7 \left[H(t) + \left|\int_{\Omega} \mathcal{V} \mathcal{V}_{t} dx\right|^{\frac{1}{1-\gamma}}+\|\Delta \mathcal{V}\|_{2}^{\frac{2}{1-\gamma}}\right].\\
\leq C_8\left[H(t)+\|\Delta \mathcal{V}\|_{2}^{2}+\|\mathcal{V}_{t}\|_{2}^{2}+\|\mathcal{V}\|_{p}^{p}\right].
\end{gathered}
\label{eq72}
\end{equation}
$$
$$
From (\ref{eq65}) and (\ref{eq72}), we get
\begin{equation}
\mathcal{A}^{\prime}(t) \geq D \mathcal{A}^{\frac{1}{1-\gamma}}(t).
\label{eq73}
\end{equation}
where the constants $C_3,C_4,C_5,C_6,C_7,C_8,D >0$. Through a straightforward integration of (\ref{eq73}), we arrive at
$$
\mathcal{A}^{\frac{\gamma}{1-\gamma}}(t) \geq \frac{1}{\mathcal{K}^{\frac{-\gamma}{1-\gamma}}(0)-D \frac{\gamma}{(1-\gamma)} t}.
$$
Therefore, $\mathcal{A}(t)$ blows-up in time
$$
T \leq T^{*}=\frac{1-\gamma}{D \gamma \mathcal{A}^{\gamma /(1-\gamma)}(0)}.
$$
The proof is now complete.\\
\end{proof}
\section{Numerical Approximation}
\label{sec:numerical}
In this section, we present a comprehensive numerical scheme to approximate solutions of the complete plate equation with fractional damping, time delay, and polynomial nonlinearity as given in \eqref{1.1}. The aim is to validate the theoretical results concerning stability, energy decay, and finite-time blow-up through computational experiments, capturing the complex interplay between fractional dissipation, delayed feedback, and nonlinear effects.

We consider the full one-dimensional version of \eqref{1.1}:
\begin{equation}
\mathcal{V}_{tt} + \mathcal{V}_{xxxx} + \partial_t^{\theta,\vartheta}\mathcal{V} + a_1 \mathcal{V}_t + a_2 \mathcal{V}_t(x,t-s) = \mathcal{V}|\mathcal{V}|^{p-2}, \quad x \in (0,L), \quad t>0,
\label{eq:full}
\end{equation}
with clamped boundary conditions:
\begin{equation}
\mathcal{V}(0,t) = \mathcal{V}(L,t) = 0, \quad \mathcal{V}_x(0,t) = \mathcal{V}_x(L,t) = 0,
\label{eq:bc}
\end{equation}
initial conditions:
\begin{equation}
\mathcal{V}(x,0) = \mathcal{V}_0(x), \quad \mathcal{V}_t(x,0) = \mathcal{V}_1(x),
\label{eq:ic}
\end{equation}
and history function for the delay term:
\begin{equation}
\mathcal{V}_t(x,t-s) = f_0(x,t-s), \quad x \in \Omega, \; t \in (0,s).
\label{eq:history}
\end{equation}
The spatial discretization is carried out using Hermite cubic finite elements, which naturally provide the required $C^1$-continuity for fourth-order problems. For temporal discretization, we employ a $\beta$-Newmark scheme that preserves the energy structure of the system. The fractional derivative is approximated using the extended variables approach following Komornik et al. \cite{KomornikSepulvedaVera2025}, which provides an energy-consistent discretization. The delay term is handled via linear interpolation that maintains consistency with the Newmark scheme.

The numerical approximation of fractional derivatives in our scheme follows the extended variables approach originally introduced by Mbodje \cite{MBO} and subsequently developed in several recent works \cite{KomornikSepulvedaVera2025, FA, AmmariKomornikSepulvedaVera2025b, AmmariKomornikSepulvedaVera2026,  AslamHajjejHaoSepulveda2025}. This method provides an energy-consistent discretization that preserves the dissipative structure of the fractional operator while maintaining computational efficiency.
\subsection{Temporal discretization: The Newmark scheme with fractional and delay terms}
\label{subsec:newmark_full}
Let $0 = t_0 < t_1 < \dots < t_M = T$ be a uniform partition of the time interval $[0,T]$ with step size $\Delta t = T/M$. Denote $\mathcal{V}^n \approx \mathcal{V}(x, t_n)$, $\mathcal{V}_t^n \approx \mathcal{V}_t(x, t_n)$, and $\mathcal{V}_{tt}^n \approx \mathcal{V}_{tt}(x, t_n)$. We assume the delay $s$ is a multiple of $\Delta t$, i.e., $s = m\Delta t$ for some integer $m \geq 1$.

The Newmark method approximates the system \eqref{eq:full} at time $t_{n+1}$ as follows:
\begin{align}
\mathcal{V}_t^{n+1} &= \mathcal{V}_t^n + (1-\gamma)\Delta t\,\mathcal{V}_{tt}^n + \gamma\Delta t\,\mathcal{V}_{tt}^{n+1}, \label{eq:newmark1_full} \\
\mathcal{V}^{n+1} &= \mathcal{V}^n + \Delta t \mathcal{V}_t^n + \left(\frac{1}{2}-\beta\right)\Delta t^2\,\mathcal{V}_{tt}^n + \beta\Delta t^2\,\mathcal{V}_{tt}^{n+1}, \label{eq:newmark2_full}
\end{align}
where $\beta$ and $\gamma$ are parameters controlling stability and numerical damping. For energy conservation properties, we choose $\beta = \frac14$, $\gamma = \frac12$ (trapezoidal rule). The equation of motion is enforced at $t_{n+1}$:
\begin{equation}
\mathcal{V}_{tt}^{n+1} + \mathcal{V}_{xxxx}^{n+1} + (\partial_t^{\theta,\vartheta}\mathcal{V})^{n+1} + a_1 \mathcal{V}_t^{n+1} + a_2 \mathcal{V}_t^{n+1-m} = \mathcal{V}^{n+1} |\mathcal{V}^{n+1}|^{p-2}. \label{eq:motion_full}
\end{equation}
\subsubsection{Discretization of the fractional derivative via extended variables}
Following Komornik et al. \cite{KomornikSepulvedaVera2025}, we approximate the fractional damping term using the extended variables approach. Let $R > 0$ be a truncation parameter and $M$ a positive integer. Define $\xi_\ell = \ell\Delta\xi$ for $\ell = 1,\dots,M$, where $\Delta\xi = R/M$, and
\begin{equation}
\mu_\ell = |\xi_\ell|^{(2\theta-1)/2}, \quad \ell=1,\dots,M, \quad 0 < \theta < 1.
\label{eq:mu_definition}
\end{equation}
Introduce auxiliary variables $\mathcal{G}_\ell^n(x) \approx \mathcal{G}(x,\xi_\ell,t_n)$ satisfying the differential equation:
\begin{equation}
\mathcal{G}_{t}(x,\xi,t) + (\xi^2 + \vartheta)\mathcal{G}(x,\xi,t) = \mu(\xi)\mathcal{V}_t(x,t).
\label{eq:auxiliary_eq}
\end{equation}
The fractional derivative is then approximated by:
\begin{equation}
(\partial_t^{\theta,\vartheta}\mathcal{V})^{n} \approx b \sum_{\ell=1}^{M} \mu_\ell \mathcal{G}_\ell^n \Delta\xi,
\label{eq:frac_approx}
\end{equation}
where $b = \frac{\sin(\theta\pi)}{\pi} a_1$ as defined in Lemma 2.1.

Discretizing \eqref{eq:auxiliary_eq} using the Crank-Nicolson scheme to preserve energy conservation, we obtain:
\begin{equation}
\mathcal{G}_\ell^{n+1} = \mathcal{G}_\ell^{n} - \Delta t (\xi_\ell^2 + \vartheta) \mathcal{G}_\ell^{n+1/2} + \Delta t \mu_\ell \mathcal{V}_t^{n+1/2},
\label{eq:auxiliary_disc}
\end{equation}
where $\mathcal{G}_\ell^{n+1/2} = \frac{1}{2}(\mathcal{G}_\ell^{n} + \mathcal{G}_\ell^{n+1})$ and $\mathcal{V}_t^{n+1/2} = \frac{1}{2}(\mathcal{V}_t^{n} + \mathcal{V}_t^{n+1})$.
\subsubsection{Discretization of the delay term}
The delayed velocity term $\mathcal{V}_t^{n+1-m}$ is approximated using linear interpolation consistent with the Newmark scheme:
\begin{equation}
\mathcal{V}_t^{n+1-m} = \mathcal{V}_t^{n-m} + \frac{\Delta t}{2}(\mathcal{V}_{tt}^{n-m} + \mathcal{V}_{tt}^{n+1-m}), \label{eq:delay_approx}
\end{equation}
which maintains second-order accuracy. For $n < m$, the history function $f_0$ provides the required past values.
\subsubsection{Complete discrete system}
Let $\mathbf{Q}^n$ denote the vector of nodal degrees of freedom at time $t_n$, and $\mathbf{G}_\ell^n$ the corresponding discretization of $\mathcal{G}_\ell^n$. The semidiscrete Galerkin formulation leads to:
\begin{equation}
\mathbf{M} \ddot{\mathbf{Q}} + a_1\mathbf{M} \dot{\mathbf{Q}} + \mathbf{K} \mathbf{Q} + b\sum_{\ell=1}^{M} \mu_\ell \mathbf{G}_\ell \Delta\xi + a_2\mathbf{M} \dot{\mathbf{Q}}^{n-m} = \mathbf{F}(\mathbf{Q}), \label{eq:semidiscrete_full}
\end{equation}
where $\mathbf{M}$ is the mass matrix, $\mathbf{K}$ the stiffness matrix for the biharmonic operator, and $\mathbf{F}(\mathbf{Q})$ the nonlinear force vector given by
\begin{equation}
\mathbf{F}(\mathbf{Q}) = \int_0^L \mathcal{V}_h |\mathcal{V}_h|^{p-2} \boldsymbol{\phi} , dx,
\label{eq:nonlinear_force_vector}
\end{equation}
with $\boldsymbol{\phi}$ being the vector of Hermite shape functions. The integral is computed element-wise using Gaussian quadrature.

Applying the Newmark formulas \eqref{eq:newmark1_full}--\eqref{eq:newmark2_full} and the fractional discretization \eqref{eq:frac_approx}--\eqref{eq:auxiliary_disc}, we obtain the complete discrete system
(see \cite{AslamHajjejHaoSepulveda2025}):
\begin{equation}
\begin{cases}
\left(\mathbf{M} + \gamma\Delta t\,\left(a_1\mathbf{M}+\mathbf{C}_{augm}\right) + \beta\Delta t^2\mathbf{K}\right)\ddot{\mathbf{Q}}^{n+1}  + a_2\mathbf{M} \dot{\mathbf{Q}}^{n+1-m} \\
\quad = -2 b\sum\limits_{\ell=1}^{M} \widetilde{\mu}_\ell \mathbf{G}_\ell^{n} \Delta\xi + \mathbf{F}(\mathbf{Q}^{n+1}) - \mathbf{K}\left(\mathbf{Q}^n + \Delta t \dot{\mathbf{Q}}^n + \left(\frac{1}{2}-\beta\right)\Delta t^2 \ddot{\mathbf{Q}}^n\right) \\
\qquad - a_1\mathbf{M} \left(
\dot{\mathbf{Q}}^n + \left( 1 - {\gamma} \right)
\delta t \ddot{\mathbf{Q}}^n
\right) 
- \mathbf{C}_{augm} \left(
2\dot{\mathbf{Q}}^n + \left( 1 -{\gamma} \right)
\delta t \ddot{\mathbf{Q}}^n
\right)
\\[8pt]
\mathbf{G}_\ell^{n+1} = \dfrac{2 - \Delta t(\xi_\ell^2 + \vartheta)}{2 + \Delta t(\xi_\ell^2 + \vartheta)} \mathbf{G}_\ell^{n} + \dfrac{2\Delta t\mu_\ell}{2 + \Delta t(\xi_\ell^2 + \vartheta)} \dot{\mathbf{Q}}^{n+1/2}, \quad \ell=1,\dots,M,
\end{cases}
\label{eq:complete_discrete}
\end{equation}
where  $\widetilde{\mu}_\ell=\dfrac{2-\delta t\left(\xi_\ell^2 +\eta\right)}{2+\delta t\left(\xi_\ell^2 +\eta\right)}
\mu_\ell$,\ $\mathbf{C}_{augm}=\Delta t b \left(\displaystyle\sum_{\ell=1}^M
\dfrac{2{\mu}_\ell^2\delta\xi}{2+\delta t\left(\xi_\ell^2 +\eta\right)}\right)\textbf{M} $,
and $\dot{\mathbf{Q}}^{n+1/2} = \frac{1}{2}(\dot{\mathbf{Q}}^{n} + \dot{\mathbf{Q}}^{n+1})$.
\subsubsection{Discrete energy decay}
Following the approach in \cite{KomornikSepulvedaVera2025}, we define the discrete energy:
\begin{align}
E_\Delta^n \ = & \frac{1}{2} (\dot{\mathbf{Q}}^n)^T \mathbf{M} \dot{\mathbf{Q}}^n + \frac{1}{2} (\mathbf{Q}^n)^T \mathbf{K} \mathbf{Q}^n + \frac{b}{2} \sum_{\ell=1}^{M} \mu_\ell \|\mathbf{G}_\ell^n\|^2 \Delta\xi 
\nonumber\\
 &+ \frac{a_2}{2} \sum_{j=n-m}^{n-1} \|\dot{\mathbf{Q}}^j\|_{\mathbf{M}}^2
 - \frac{1}{p} \int_0^L |\mathcal{V}_h^n|^p dx,
\label{eq:discrete_energy}
\end{align}
where $\|\cdot\|_{\mathbf{M}}$ denotes the norm induced by the mass matrix,
and where the last term represents the discrete potential energy associated with the nonlinearity. The scheme \eqref{eq:complete_discrete} satisfies the discrete energy decay estimate:
\begin{equation}
\begin{aligned}
E_\Delta^{n+1} - E_\Delta^n &= -b \sum_{\ell=1}^{M} (\xi_\ell^2 + \vartheta) \|\mathbf{G}_\ell^{n+1/2}\|^2 \Delta\xi - a_1 \|\dot{\mathbf{Q}}^{n+1/2}\|_{\mathbf{M}}^2 \\
&\quad - a_2 \left( \dot{\mathbf{Q}}^{n+1/2} \right)^T \mathbf{M} \dot{\mathbf{Q}}^{n+1-m} \Delta t \\
&\quad + \left( \mathbf{F}(\mathbf{Q}^{n+1}) \right)^T \dot{\mathbf{Q}}^{n+1/2} \Delta t - \left( \frac{1}{p} \int_0^L |\mathcal{V}_h^{n+1}|^p dx - \frac{1}{p} \int_0^L |\mathcal{V}_h^{n}|^p dx \right) \\
&\quad + \frac{a_2}{2} \left( \|\dot{\mathbf{Q}}^{n}\|_{\mathbf{M}}^2 - \|\dot{\mathbf{Q}}^{n-m}\|_{\mathbf{M}}^2 \right).
\end{aligned}
\label{eq:energy_decay_raw}
\end{equation}
If the nonlinear term is discretized consistently with the potential, i.e.,
\begin{equation}
\left( \mathbf{F}(\mathbf{Q}^{n+1}) \right)^T \dot{\mathbf{Q}}^{n+1/2} \Delta t = \frac{1}{p} \int_0^L |\mathcal{V}_h^{n+1}|^p dx - \frac{1}{p} \int_0^L |\mathcal{V}_h^{n}|^p dx + \mathcal{O}(\Delta t^2),
\label{eq:nonlinear_consistency}
\end{equation}
then the nonlinear contributions cancel up to higher-order terms. Moreover, under a stability condition on $a_2$ (e.g., $a_2$ sufficiently small), the delay term does not cause energy growth. Consequently, we obtain the energy decay estimate:
\begin{equation}
E_\Delta^{n+1} - E_\Delta^n \leq -b \sum_{\ell=1}^{M} (\xi_\ell^2 + \vartheta) |\mathbf{G}\ell^{n+1/2}|^2 \Delta\xi - a_1 |\dot{\mathbf{Q}}^{n+1/2}|{\mathbf{M}}^2 \leq 0,
\label{eq:energy_decay}
\end{equation}
which mimics the continuous energy decay \eqref{eq2.12} and ensures numerical stability.
\subsection{Spatial discretization: Hermite cubic finite elements}
\label{subsec:hermite}
For the spatial discretization of the biharmonic operator 
$\Delta^2\mathcal{V}$, we utilize Hermite cubic finite elements. 
These elements naturally enforce $C^1$ continuity across element 
boundaries, which is essential for the weak formulation involving 
second derivatives \cite{BrennerScott08,Reddy07}. The standard 
Hermite cubic basis functions on the reference element $[0,1]$  follow the 
classical formulation \cite{StrangFix73,Hughes87}.
\subsubsection{Mesh and basis functions}
Let the domain $[0,L]$ be divided into $N$ uniform elements of length $h = L/N$. The nodes are located at $x_i = (i-1)h$, $i=1,\dots,N+1$. At each node $x_i$, we associate two degrees of freedom: the value of the solution $q_i = \mathcal{V}(x_i)$ and the value of its first derivative $r_i = \mathcal{V}_x(x_i)$. Thus, the total number of degrees of freedom is $2(N+1)$.

On the reference element $[0,1]$ with coordinate $\xi = (x - x_i)/h$, the Hermite cubic shape functions are:
\begin{align}
\phi_0(\xi) &= 1 - 3\xi^2 + 2\xi^3, \label{eq:phi0} \\
\phi_1(\xi) &= h\big(\xi - 2\xi^2 + \xi^3\big), \label{eq:phi1} \\
\phi_2(\xi) &= 3\xi^2 - 2\xi^3, \label{eq:phi2} \\
\phi_3(\xi) &= h\big(-\xi^2 + \xi^3\big). \label{eq:phi3}
\end{align}
These satisfy:
\begin{align*}
\phi_0(0)=1,\; \phi_0'(0)=0,\; \phi_0(1)=0,\; \phi_0'(1)=0; & \quad
\phi_1(0)=0,\; \phi_1'(0)=1,\; \phi_1(1)=0,\; \phi_1'(1)=0; \\
\phi_2(0)=0,\; \phi_2'(0)=0,\; \phi_2(1)=1,\; \phi_2'(1)=0; & \quad
\phi_3(0)=0,\; \phi_3'(0)=0,\; \phi_3(1)=0,\; \phi_3'(1)=1.
\end{align*}
The approximation on element $e = [x_i, x_{i+1}]$ is:
\begin{equation}
\mathcal{V}_h^e(\xi,t) = q_i \phi_0(\xi) + r_i \phi_1(\xi) + q_{i+1} \phi_2(\xi) + r_{i+1} \phi_3(\xi). \label{eq:approx_element}
\end{equation}
\subsubsection{Weak formulation}
The weak form of \eqref{2..1} in his one-dimensional version is obtained by multiplying by a test function $\psi \in H_0^2(0,L)$ and integrating by parts twice:
\begin{align*}
&\int_0^L \mathcal{V}_{tt} \psi \, dx + \int_0^L \mathcal{V}_{xx} \psi_{xx} \, dx 
+ b\int_0^L \int_{-\infty}^{+\infty}\mathscr{G}(x,\alpha,t)\beta(\alpha)
\psi \,d\alpha
 dx \nonumber
\\
& \qquad\qquad + a_1 \int_0^L \mathcal{V}_t \psi \, dx 
+  \int_0^L a_2z(x,1,t)\psi \, dx
= \int_0^L \mathcal{V}|\mathcal{V}|^{p-2} \psi \, dx. 
\\
& \int_0^L\mathscr{G}_{t}(x, \alpha, t)\psi \, dx+
\left(\alpha^{2}+\vartheta\right) \int_0^L\mathscr{G}(x, \alpha, t)\psi \, dx
\nonumber\\
&\qquad\qquad - \beta(\alpha)\int_0^L z(x, 1, t) \beta(\alpha)\psi \, dx=0 
\\
& \int_0^L \left(s z_{t}(x, \varrho, t)+z_{\varrho}(x, \varrho, t)\right)\psi \, dx =0 
\end{align*}
Substituting $\mathcal{V}_h$, $\mathscr{G}_h$, $z_h$ and $\psi = \phi_j$ (one of the basis functions) leads to the discrete system \eqref{eq:semidiscrete_full}.
\subsubsection{Matrix assembly}
The mass matrix $\mathbf{M}$ and stiffness matrix $\mathbf{K}$ are assembled elementwise. For a single element $e$, the element matrices are:
\begin{align}
M_{ij}^e &= \int_0^1 \phi_i(\xi) \phi_j(\xi) \, h \, d\xi, \label{eq:M_element} \\
K_{ij}^e &= \int_0^1 \phi_i''(\xi) \phi_j''(\xi) \, \frac{1}{h^3} \, d\xi, \label{eq:K_element}
\end{align}
where $\phi_i''$ denotes the second derivative with respect to $x$. Using \eqref{eq:phi0}--\eqref{eq:phi3}, these integrals can be computed analytically or by numerical quadrature. The global matrices are obtained by summing contributions from all elements, respecting the connectivity of degrees of freedom.

The nonlinear force vector $\mathbf{F}(\mathbf{Q})$ is computed via numerical integration (e.g., Gaussian quadrature) on each element:
\begin{equation}
F_i = \int_0^L \mathcal{V}_h |\mathcal{V}_h|^{p-2} \phi_i \, dx. \label{eq:nonlinear_force}
\end{equation}
\subsubsection{Boundary conditions}
Clamped boundary conditions \eqref{eq:bc} are imposed by setting:
\[
q_1 = r_1 = 0, \quad q_{N+1} = r_{N+1} = 0,
\]
and eliminating these degrees of freedom from the system.
\subsection{Algorithm summary}
\label{subsec:algorithm}
The complete numerical algorithm for solving \eqref{eq:full}--\eqref{eq:ic} is as follows:
\begin{algorithm}[htbp]
\caption{Complete Numerical Scheme for Plate Equation with Fractional Damping, Delay, and Nonlinearity}
\label{alg:full_numerical_scheme}
\SetAlgoLined
\DontPrintSemicolon
\KwIn{Parameters: $L, T, N, M, \Delta t, \Delta\xi, a_1, a_2, \theta, \vartheta, p$,\\
Initial conditions $\mathcal{V}_0$, $\mathcal{V}_1$,\\
History function $f_0$}
\KwOut{Discrete solution $\{\mathbf{Q}^n, \dot{\mathbf{Q}}^n\}_{n=0}^{N_t}$, energy $\{E_\Delta^n\}_{n=0}^{N_t}$}

\textbf{Initialization:}\;
Construct $\mathbf{M}$ and $\mathbf{K}$ using Hermite cubic finite elements\;
Interpolate initial conditions: $\mathbf{Q}^0 \gets \Pi_h\mathcal{V}_0$, $\dot{\mathbf{Q}}^0 \gets \Pi_h\mathcal{V}_1$\;
Initialize auxiliary variables: $\mathbf{G}_\ell^0 \gets \mathbf{0}$ for $\ell = 1,\dots,M$\;
Compute $\mu_\ell = |\xi_\ell|^{(2\theta-1)/2}$ for $\ell = 1,\dots,M$\;
Initialize delay manager with $s = m\Delta t$ and history $f_0$\;
Initialize buffer for velocity norms and cumulative sum $S \gets 0$\;
Compute initial energy $E_\Delta^0$ using \eqref{eq:discrete_energy}\;

\For{$n \gets 0$ \KwTo $N_t-1$}{
    Get delayed velocity: $\dot{\mathbf{Q}}^{n+1-m} \gets \text{DelayManager.get}(\dot{\mathbf{Q}}^n)$\;
    
    \textbf{Solve nonlinear system} for $\ddot{\mathbf{Q}}^{n+1}$:\;
    \While{not converged}{
        Compute nonlinear force: $\mathbf{F}(\mathbf{Q}^{n+1})$\;
        Assemble RHS using \eqref{eq:complete_discrete}\;
        Solve $\mathbf{A}\ddot{\mathbf{Q}}^{n+1} = \mathbf{b}_{NL}$\;
        Update $\mathbf{Q}^{n+1}$, $\dot{\mathbf{Q}}^{n+1}$ via Newmark formulas\;
    }
    
    \textbf{Update auxiliary variables}:\;
    \For{$\ell \gets 1$ \KwTo $M$}{
        $\mathbf{G}_\ell^{n+1} \gets \dfrac{2 - \Delta t(\xi_\ell^2 + \vartheta)}{2 + \Delta t(\xi_\ell^2 + \vartheta)} \mathbf{G}_\ell^{n} + \dfrac{2\Delta t\mu_\ell}{2 + \Delta t(\xi_\ell^2 + \vartheta)} \dot{\mathbf{Q}}^{n+1/2}$\;
    }
    
    \textbf{Update energy with delay term}:\;
    Compute current norm: $N_{\text{curr}} \gets (\dot{\mathbf{Q}}^{n+1})^T\mathbf{M}\dot{\mathbf{Q}}^{n+1}$\;
    Update cumulative sum $S$ using circular buffer\;
    Compute $E_\Delta^{n+1}$ including $\frac{a_2}{2}S$ term\;
    
    \textbf{Update delay manager} with $\dot{\mathbf{Q}}^{n+1}$\;
}

\Return{$\{\mathbf{Q}^n, \dot{\mathbf{Q}}^n, E_\Delta^n\}_{n=0}^{N_t}$}
\end{algorithm}

This numerical framework provides a reliable tool to simulate the behavior of the nonlinear plate equation, allowing us to observe both exponential decay (for positive damping and bounded initial energy) and finite-time blow-up (for negative initial energy), thereby corroborating the analytical findings of Sections \ref{S5} and \ref{sec6}.

\subsection{Numerical examples}

Throughout the following examples, we fix the domain length $L=1$ and choose the parameters
\[
a_1 = 1.0,\quad a_2 = 0.08,\quad s = 5,\quad \theta = 0.5,\quad \vartheta = 0.3,
\]
which satisfy condition (A1). For the spatial discretization we use a uniform mesh of $N=250$ nodes, resulting in $N_e=249$ elements and $N_h=496$ degrees of freedom in the Hermite polynomial basis.

The initial data are chosen as  
\[
\mathcal{V}_0(x) = \lambda\,x^2(1-x)^2,\qquad 
\mathcal{V}_1(x) = 0,\qquad 
f_0(x,t-s)=0,
\qquad 0<x<L,\;0<t<s,
\]  
so that \(U_0 \in D(A)\). The parameter \(\lambda\) controls the initial energy \(E(0)\). According to Theorem \ref{T59}, if \(0 < E(0) < d\), where  
\[
d = \frac{p-2}{2p}\left(C^*_p\right)^{\frac{2}{2-p}},
\]  
with $ C^*_p=\dfrac{L^2}{\pi^2}$
and condition \eqref{eq4.1} holds, then the energy decays exponentially. Conversely, if \(E(0) < 0\), finite‑time blow‑up may occur.

Table~\ref{table1} lists, for several values of \(p\), the critical values \(\lambda_c\) (where \(E(0)=0\)) and \(\lambda_d\) (where \(E(0)=d\)). For a fixed \(p\), the solution’s asymptotic behavior depends on \(\lambda\) as follows:  
\begin{itemize}
    \item If \(\lambda < \lambda_d(p)\), then \(E(0) > d\); the hypotheses of Theorem \ref{T59} are not satisfied, and the asymptotic behavior is not determined by that theorem.
    \item If \(\lambda_d(p) < \lambda < \lambda_c(p)\), then \(0 < E(0) < d\); exponential decay is expected.
    \item If \(\lambda > \lambda_c(p)\), then \(E(0) < 0\); finite‑time blow‑up may occur.
\end{itemize}
\begin{table}[htbp]
\centering
\label{tab:lambda_critical}
\begin{tabular}{c|c|c|c}
$p$ & $\lambda_c$ & $ d=\frac{p-2}{2p}\left(C^*_p\right)^{\frac{2}{2-p}}$ & $\lambda_d$\\ \hline
3 & 14414.4 & 16.2348 & 6.372 \\
4 & 591.66 & 2.4674	& 2.484 \\
5 & 198.15 & 1.3788	& 1.8567 \\
6 & 112.87 & 1.0472	& 1.6180\\
7 & 79.90 & 0.8467	 & 1.455\\
8 & 63.15 & 0.806	& 1.419\\
9 & 53.21 & 0.688	& 1.311\\
\end{tabular}
\caption{Critical values of $\lambda$ for which $E(0)=0$ and $E(0)=d$}
\label{table1}
\end{table}
In the simulations below we select specific values of \(\lambda\) and \(p\) to illustrate both the exponential decay and the blow‑up regimes.
\subsubsection{Example 1: Exponential stabilty.}
We set the parameters $\lambda=1$ and $p=5$, yielding an initial energy of $E(0)\approx 0.400127$, which is below the critical threshold $E_c \approx 1.3788$ (see Table \ref{table1}). This configuration satisfies the hypotheses of Theorem~\ref{T59}. In addition, we chose $\theta=0.5$, $\vartheta=0.3$,
$a_1=5.0$, and $a_2=0.4$
 to verify (A1).

Figure \ref{fig0} demonstrates the exponential decay of the total energy (right panel) and the exponential stability of the solution. The displacement of the beam, represented by $\mathcal{V}(x,t)$, decays exponentially over time (left panel). A transient energy increase is observed during the initial 5 time units, prior to the activation of the delay term. For $t>5$, the energy exhibits clear exponential decay, as evidenced by the nearly constant slope in the semi-logarithmic plot.

\begin{figure}[htbp]
    \centering
    \includegraphics[width=0.49\linewidth]{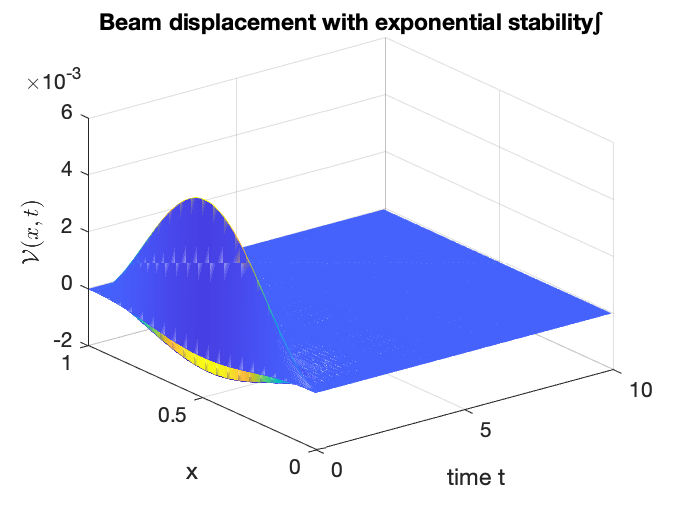}
    \includegraphics[width=0.49\linewidth]{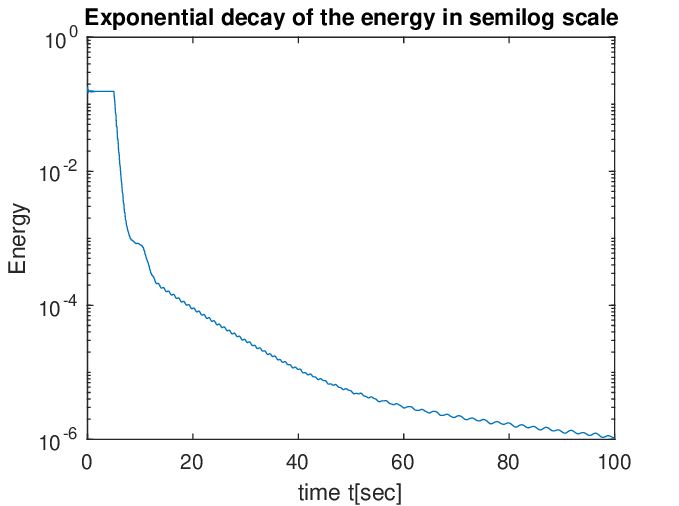}
    \caption{Exponential stability of the solution. Left: Evolution of displacement $\mathcal{V}(x,t)$. Right: Energy decay in semi-logarithmic scale.}
    \label{fig0}
\end{figure}
\subsubsection{Example 2: Finite-Time Blow-Up.}
Choosing $\lambda=200$ and $p=5$, we obtain the initial energy 
$E(0)\approx -496.4864 <0$, which satisfies the assumptions of 
Theorem~\ref{thm6.2} (see Table \ref{table1}). Consequently, the solution exhibits a finite-time 
blow-up. Numerical simulations confirm this behavior, showing a blow-up 
occurring very close to $t=0.04$; see Figure~\ref{fig1} for the rapid 
growth of the energy near that time.

\begin{figure}[htbp]
    \centering
    \includegraphics[width=0.5\linewidth]{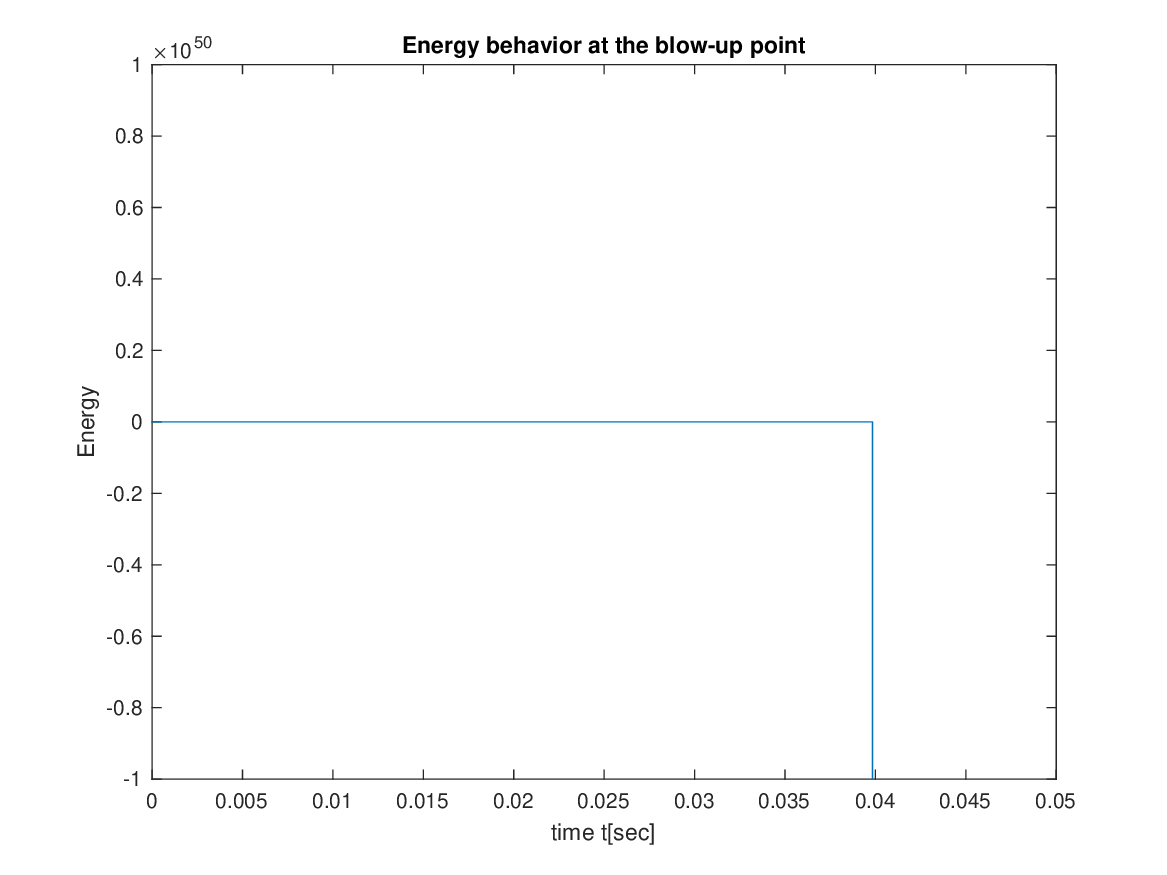}
    \caption{Energy growth near the blow-up time $t\approx 0.04$.}
    \label{fig1}
\end{figure}
Figure~\ref{fig2} displays the displacement profile immediately before 
the blow-up. The left panel shows the spatio-temporal evolution of the 
solution from the initial time until approximately $10^{-4}$ seconds 
before the blow-up. The right panel focuses on the last $4\times10^{-4}$ 
seconds. The extreme gradient developed during this brief interval 
illustrates the intrinsic difficulty of capturing the blow-up 
numerically; indeed, our scheme fails to compute the solution beyond 
this point.
\begin{figure}[htbp]
    \centering
    \includegraphics[width=0.49\linewidth]{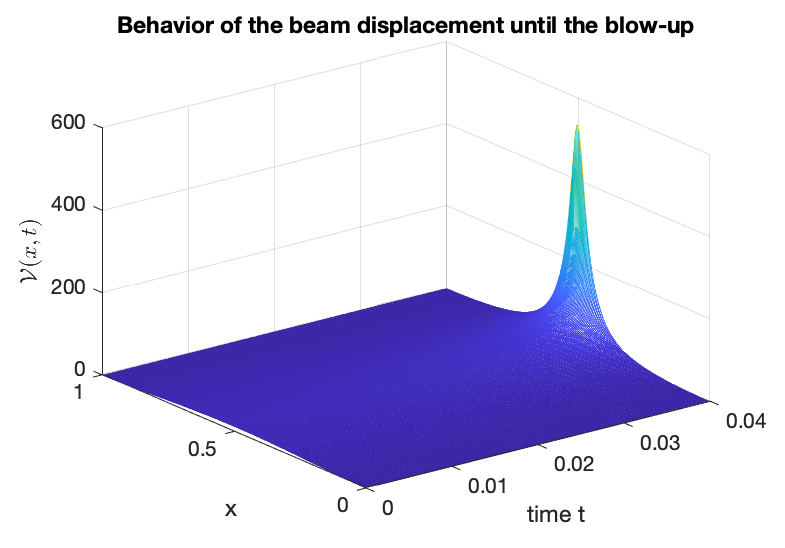}
    \includegraphics[width=0.49\linewidth]{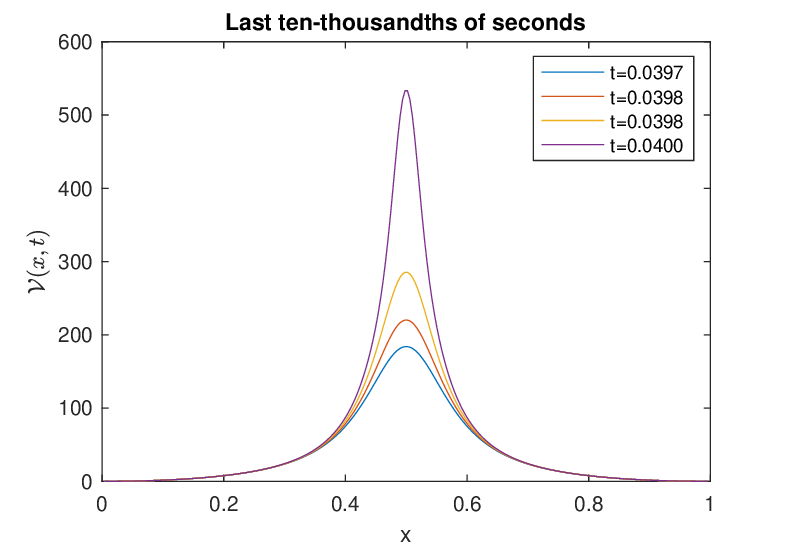}
    \caption{Beam displacement shortly before blow-up. Left: full 
             evolution until $t\approx 0.04$. Right: detailed view of the 
             last $4\times10^{-4}$ seconds.}
    \label{fig2}
\end{figure}
\subsubsection{Example 3: Uncertain behavior under insufficient stability conditions}

\begin{figure}[htbp]
    \centering
    \includegraphics[width=0.49\linewidth]{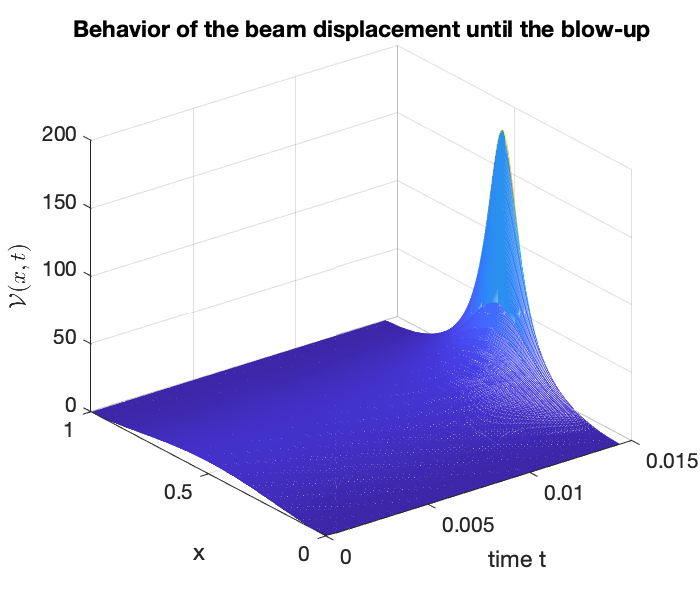}
    \includegraphics[width=0.49\linewidth]{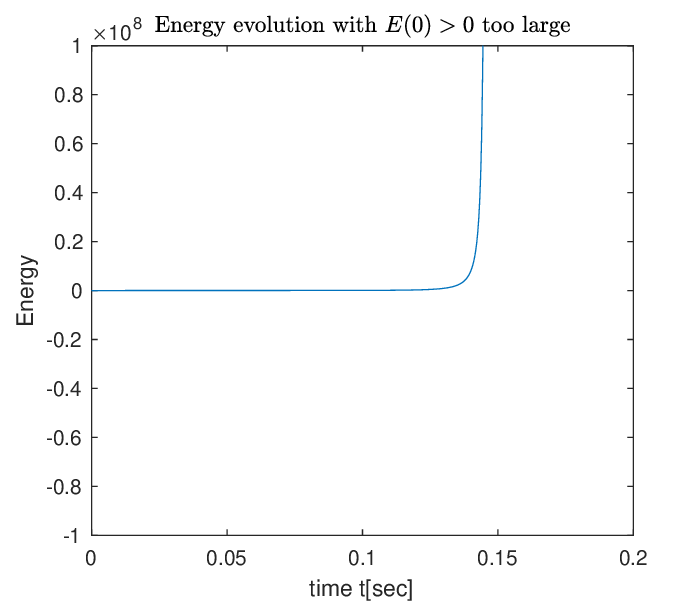}
    \caption{Violation of condition \eqref{eq4.1}. Left: displacement evolution until blow-up ($t \approx 0.04$). Right: corresponding energy evolution.}
    \label{fig3}
\end{figure}
This example demonstrates the dynamical response when neither the exponential stability conditions nor the blow-up condition are satisfied, leading to uncertain system behavior.

\begin{figure}[htbp]
    \centering
    \includegraphics[width=0.49\linewidth]{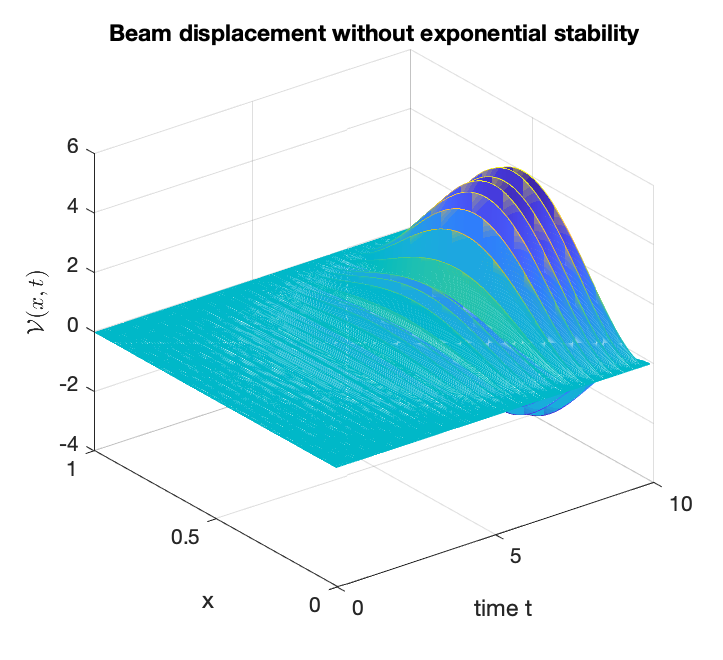}
    \includegraphics[width=0.49\linewidth]{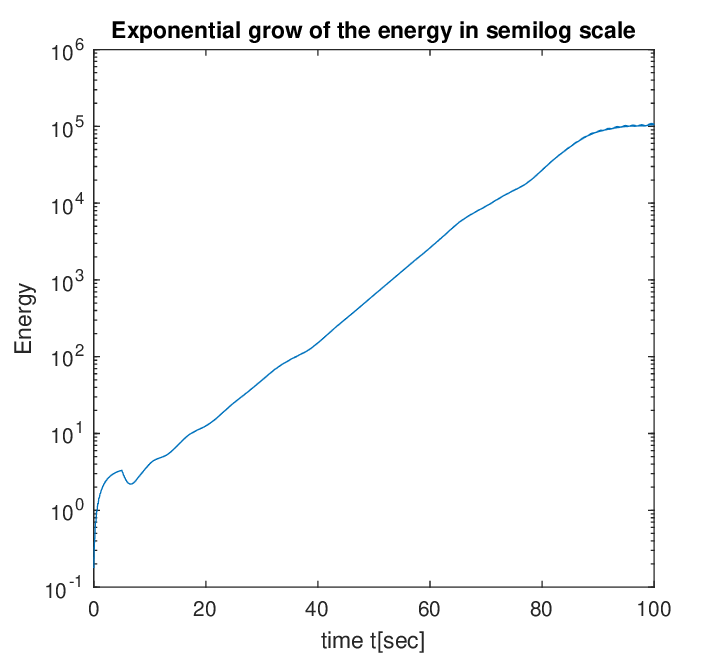}
    \caption{Violation of condition (A1). Left: displacement evolution. Right: exponential energy growth without blow-up.}
    \label{fig4}
\end{figure}
In the first scenario (Figure \ref{fig3}), we select parameters violating condition \eqref{eq4.1} (see Table \ref{table1}). The potential energy is $\frac{1}{2}\|\mathcal{V}_{xx}(\cdot,0)\|^2 = 8325.866906$, while the nonlinear energy contribution is $\frac{1}{5}\int_0^1|\mathcal{V}_0|^5 dx = 3206.641840$, yielding a positive initial energy $E(0) \approx 5119.225066$. Although $E(0) > 0$, a finite-time blow-up occurs just before $t = 0.15$, confirming that condition \eqref{eq4.1} is necessary to prevent blow-up even with positive initial energy.

The second scenario (Figure \ref{fig4}) examines violation of condition (A1) while maintaining $E(0) > 0$. Using parameters from Example 1 but with $a_1 = 1$ and $a_2 = 2$, condition (A1) is not satisfied. In this case, no blow-up is observed within $t = 100$ time units; instead, the energy exhibits exponential growth, indicating instability despite the positive initial energy.

Both cases illustrate uncertain dynamical behavior: the system neither achieves exponential stability (since the sufficient conditions are violated) nor exhibits the guaranteed blow-up predicted for $E(0) < 0$. These examples highlight the sensitivity of the system's long-term behavior to the precise satisfaction of the theoretical conditions.
\section*{Declarations}
\textbf{Ethical Approval} \\
Not applicable \\
\textbf{Availability of Data and Materials} \\
All data are available in manuscript. \\
\textbf{Funding} \\
M. Sepúlveda thanks Fondecyt-ANID project 1220869, and the support of
ANID-Chile through Centro de Modelamiento Matemático (FB210005).
\\
\textbf{Authors' Contributions} \\
All authors have approved the final version of the manuscript. \\
\textbf{Competing Interests} \\
The authors declare that they have no competing interests. \\
\textbf{Acknowledgments} \\
All authors have agreed to submit this version. 

\end{document}